\documentclass[11pt]{article}%
\usepackage{amssymb}
\usepackage{amsfonts}
\usepackage{amsmath, amsthm, amssymb}
\usepackage{amsmath}
\usepackage{graphicx}%
\newtheorem{theorem}{Theorem}
\newtheorem{theoremA}{Theorem}

\newtheorem{conjecture}{Conjecture}

\newtheorem{lemma}{Lemma}

\newtheorem{problem}{Problem}
\newtheorem{proposition}{Proposition}
\newtheorem{remark}{Remark}

\begin{document}

\title{Point-sets in general position with many \\similar copies of a pattern \footnotetext{Some preliminary results appeared in \cite{AF99}}}
\author{Bernardo M. \'{A}brego and Silvia Fern\'{a}ndez-Merchant\\{\small Department of Mathematics}\\{\small California State University, Northridge,}\\{\small 18111 Nordhoff St, Northridge, CA 91330-8313.}\\{\footnotesize {email:\texttt{\{bernardo.abrego, silvia.fernandez
\}@csun.edu}}}}
\date{\today, v.17}
\maketitle

\begin{abstract}
For every pattern $P$, consisting of a finite set of points in the
plane, $S_{P}(n,m)$ is defined as the largest number of similar
copies of $P$ among sets of $n$ points in the plane without $m$
points on a line. A general construction, based on iterated
Minkovski sums, is used to obtain new lower bounds for $S_{P}(n,m)$
when $P$ is an arbitrary pattern. Improved bounds are obtained when
$P$ is a triangle or a regular polygon with few sides. It is also
shown that $S_{P}(n,m)\geq n^{2-\varepsilon}$ whenever
$m(n)\rightarrow \infty$ as $n \rightarrow\infty$. Finite sets with
no collinear triples and not containing the 4 vertices of any
parallelogram are called \emph{parallelogram-free}. The more
restricted function $S_{P} ^{\nparallel }(n)$, defined as the
maximum number of similar copies of $P$ among parallelogram-free
sets of $n$ points, is also studied. It is proved that $\Omega(n\log
n)\leq S_{P}^{\nparallel}(n)\leq O(n^{3/2})$.\medskip


\end{abstract}

%\keywords{ similar copy, pattern, general position, collinear points, parallelogram-free, Minkovski Sum.}

\section{Introduction}

Sets $A$ and $B$ in the plane are \emph{similar}, denoted by $A\sim B$, if
there is an orientation-preserving isometry followed by a dilation that takes
$A$ to $B$. Identifying the plane with $\mathbb{C}$, the set of complex
numbers, $A\sim B$ if there are complex numbers $w$ and $z\neq0$ such that
$B=zA+w$. Here, $zA=\{za:a\in A\}$ and $A+w=\{a+w:a\in A\}$.

Consider a finite set of points $P$ in the plane, $|P|\geq3$. We refer to $P$
as a \emph{pattern} ($P$ is usually fixed). For any finite set of points $Q$,
we define $S_{P}\left(  Q\right)  $ to be the number of similar copies of $P$
contained in $Q$. More precisely,%
\[
S_{P}\left(  Q\right)  =\left\vert \left\{  P^{\prime}\subseteq Q:P^{\prime
}\sim P\right\}  \right\vert .
\]
The main goal of this paper is to explicitly construct point sets $Q$ in
\emph{general position} with a large number of similar copies of the pattern
$P$, that is, with large $S_{P}(Q)$. By general position we mean to forbid
triples of collinear points, although we also consider the restriction of
allowing at most $m$ points on a line, $m\geq3$, and the stronger restriction
of not allowing collinear points or parallelograms in $Q$.%
%TCIMACRO{\FRAME{fhFU}{3.3745in}{3.5561in}{0pt}{\Qcb{Point-set in general
%position with many triples spanning equilateral triangles.}}{\Qlb{fig:
%minkovski}}{minkov.eps}{\special{ language "Scientific Word";
%type "GRAPHIC";  maintain-aspect-ratio TRUE;  display "USEDEF";
%valid_file "F";  width 3.3745in;  height 3.5561in;  depth 0pt;
%original-width 3.3451in;  original-height 3.5284in;  cropleft "0";
%filename 'Minkov.eps';file-properties "XNPEU";}}}%
%BeginExpansion
\begin{figure}
[h]
\begin{center}
\includegraphics[
height=3.5561in,
width=3.3745in
]%
{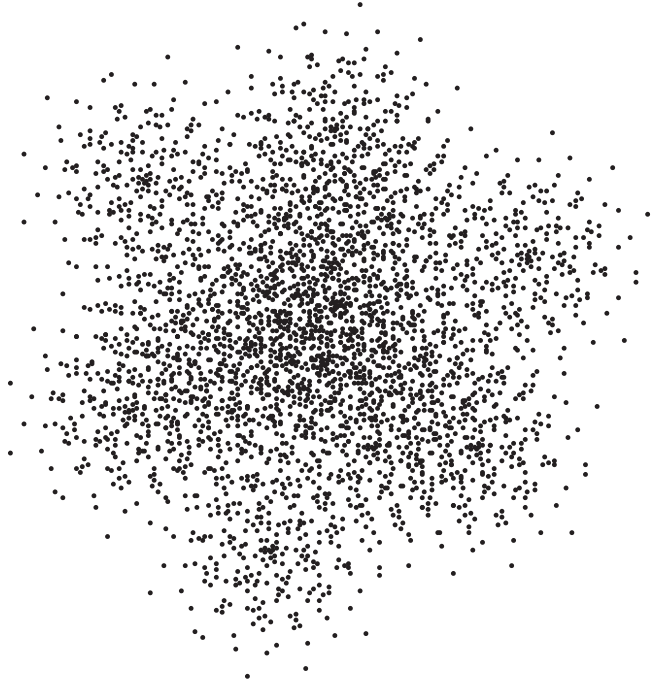}%
\caption{Point-set in general position with many triples spanning equilateral
triangles.}%
\label{fig: minkovski}%
\end{center}
\end{figure}
%EndExpansion

To explain the motivation of this paper we first turn to the original problem.
Erd\H{o}s and Purdy \cite{EP71}-\cite{EP76} posed the problem of maximizing
the number of similar copies of $P$ contained in a set of $n$ points in the
plane. That is, to determine the function%
\[
S_{P}\left(  n\right)  =\max_{\left\vert Q\right\vert =n}S_{P}\left(
Q\right)  ,
\]
where the maximum is taken over all point-sets $Q\subseteq\mathbb{C}$ with $n$
points. Elekes and Erd\H{o}s \cite{EE94} noted that $S_{P}\left(  n\right)
\leq n\left(  n-1\right)  $ for any pattern $P$. They gave a quadratic lower
bound for $S_{P}(n)$ when $\left\vert P\right\vert =3$ or when all the
coordinates of the points in $P$ are algebraic numbers. They also proved a
slightly subquadratic lower bound for all other patterns $P$. (The precise
statements can be found in Section \ref{sec: when m grows}.) Later, Laczkovich
and Ruzsa \cite{LR97} characterized precisely those patterns for which
$S_{P}\left(  n\right)  =\Theta(n^{2})$. However, the coefficient of the
quadratic term is not known for any non-trivial pattern, not even for the
simplest case of $P$ an equilateral triangle \cite{AF00}. Elekes and the
authors \cite{AFE04} investigated the structural properties of the $n$-sets
$Q$ that achieve a quadratic number of similar copies of a pattern $P$. We
proved that those sets contain large lattice-like structures, and therefore,
many collinear points. In particular, the next result was obtained.

\begin{theoremA}
[ \'{A}brego et al. \cite{AFE04}]\label{th:Abregoetal} For every positive
integer $m$ and every real $c>0$, there is a threshold function $N_{0}%
=N_{0}(c,m)$ with the following property: if $n\geq N_{0}$ and $Q$ is an
$n$-set with $S_{P}(Q)\geq cn^{2}$, then $Q$ has $m$ points on a line forming
an arithmetic progression.
\end{theoremA}

Thus having many points on a line is a required property to achieve
$\Theta(n^{2})$ similar copies of a pattern $P$. It is only natural to
restrict the problem of maximizing $S_{P}(Q)$ over sets $Q$ with a limited
number of possible collinear points. Given natural numbers $n$ and $m\geq3$,
we restrict the maximum in $S_{P}\left(  n\right)  $ to $n$-sets with at most
$m-1$ collinear points. We denote this maximum by%
\[
S_{P}\left(  n,m\right)  =\max\left\{  S_{P}\left(  Q\right)  :\left\vert
Q\right\vert =n\text{ and no }m\text{ points of }Q\text{ are collinear}%
\right\}  .
\]
We mention here that Erd\H{o}s' unit distance problem \cite{Er} (arguably the
most important problem in the area), has also been studied under general
position assumptions, like no 3 points on a line \cite{EP76} and also no
parallelograms \cite{SV} (see also \cite[Sections 5.1 and 5.5]{BrMoPa}).

When the maximum of $S_{P}(Q)$ is taken over $n$-sets in \emph{general
position}, that is, \emph{ }no three collinear points, we obtain the function
$S_{P}\left(  n,3\right)  $, which we simply denote by $S_{P}^{\prime}\left(
n\right)  $. By definition, it is clear that $S_{P}^{\prime}\left(  n\right)
\leq S_{P}\left(  n,m_{1}\right)  \leq S_{P}\left(  n,m_{2}\right)  \leq
S_{P}\left(  n\right)  $ whenever $m_{1}\leq m_{2}$. Moreover, if $m$ is
constant, then Theorem \ref{th:Abregoetal} implies that $\lim_{n\rightarrow
\infty}S_{P}(n,m)/n^{2}=0$, i.e., $S_{P}\left(  n,m\right)  =o\left(
n^{2}\right)  $. We believe that the true asymptotic value of $S_{P}(n,m)$ is
close to quadratic; however, prior to this work there were no lower bounds
other than the trivial $S_{P}^{\prime}(n)=\Omega(n)$. The rest of the paper is
devoted to the construction of point-sets giving non-trivial lower bounds for
$S_{P}(n,m)$.

\section{Statement of results}

The symmetries of the pattern $P$ play an important role in the order of
magnitude of our lower bound to the function $S_{P}\left(  n,m\right)  $. Let
us denote by $\mathrm{Iso}^{+}(P)$ the group of orientation-preserving
isometries of the pattern $P$, also known as the \emph{proper symmetry group}
of $P$. Define the \emph{index }of a point set $A$ with respect to the pattern
$P$, denoted by $i_{P}(A)$, as%
\[
i_{P}(A)=\frac{\log\left(  \left\vert \mathrm{Iso}^{+}(P)\right\vert
S_{P}(A)+\left\vert A\right\vert \right)  }{\log\left\vert A\right\vert
}\text{.}%
\]
Observe that $1\leq i_{P}(A)\leq2$, because $\left\vert \mathrm{Iso}%
^{+}(P)\right\vert S_{P}(A)\leq\left\vert A\right\vert ^{2}-\left\vert
A\right\vert $ as was noted by Elekes and Erd\H{o}s, and moreover,
$i_{P}(A)=1$ if and only if $S_{P}(A)=0$. Our main theorem gives a lower bound
for $S_{P}\left(  n,m\right)  $ using the index as the corresponding exponent.
Its proof, presented in Section \ref{sec: Main}, is based on iterated
Minkovski Sums.

\begin{theorem}
\label{th:main}For any $A$ and $P$ finite sets in the plane with at most $m-1
$ collinear points, there is a constant $c=c(P,A)$ such that, for $n$ large
enough,%
\[
S_{P}(n,m)\geq cn^{i_{P}(A)}.
\]

\end{theorem}

By using $A=P$, we conclude that the function $S_{P}\left(  n,m\right)  $ is
superlinear, that is for any finite pattern $P$ with at most $m-1$ collinear
points, $S_{P}(n,m)\geq\Omega(n^{\log(1+\left\vert P\right\vert )/\log
\left\vert P\right\vert })$. So the key to obtaining good lower bounds for
these functions, is to begin with a set $A$ with large index. For a general
pattern $P$, we can marginally improve the last inequality by constructing a
better initial set $A$. The proof of next theorem is in Section
\ref{sec: arbitrary}.

\begin{theorem}
\label{th: general P}For any finite pattern $P$ with at most $m-1$ collinear
points and $\left\vert P\right\vert =k \geq3$, there is a constant $c=c(P)$
such that, for $n$ large enough,
\[
S_{P}(n,m)\geq cn^{\log(k^{2}+k)/\log\left(  k^{2}-k+1\right)  }.
\]
\bigskip
\end{theorem}

The following theorems summarize the lower bounds for
$S_{P}^{\prime}\left( n\right)  $ obtained from the best known
initial sets $A$ for some specific patterns $P$. We concentrate on
triangles and regular polygons. We often refer to a finite pattern
as a geometric figure. For instance, when we say \textquotedblleft
let $P$ be the equilateral triangle\textquotedblright\ we actually
mean the set of vertices of an equilateral triangle. \newpage
\begin{theorem}
\label{th:initialsetstriang}Let $P=T$ be a triangle.
\begin{description}
\item If $T=\triangle$ is equilateral, then $S_{\triangle}^{\prime}\left(
n\right)  \geq\Omega\left(  n^{\log102/\log15}\right)  \geq\Omega\left(
n^{1.707}\right)  $.

\item If $T$ is isosceles, then $S_{T}^{\prime}\left(  n\right)  \geq
\Omega\left(  n^{\log17/\log8}\right)  \geq\Omega\left(  n^{1.362}\right)  $.

\item If $T$ is almost any scalene triangle, then $S_{T}^{\prime}\left(
n\right)  \geq\Omega\left(  n^{\log40/\log14}\right)
\geq\Omega\left( n^{1.397}\right)  $. For all others,
$S_{T}^{\prime}\left(  n\right) \geq\Omega\left(
n^{\log9/\log5}\right)  \geq\Omega\left(  n^{1.365}\right) $.
\end{description}
\end{theorem}

In general, if $P$ is a $k$-sided regular polygon, then $i_{P}(P)=\log\left(
2k\right)  /\log k$ and thus $S_{P}(n)\geq\Omega\left(  n^{\log\left(
2k\right)  /\log k}\right)  $. If $k$ is even, $4\leq k\leq10$ or if $k=5$ we
have the following improvement.

\begin{theorem}
\label{th:initialsetsregpoly}Let $P=R(k)$ be a regular $k$-gon. Then%
\[%
\begin{array}
[c]{ll}%
S_{R(4)}^{\prime}\left(  n\right)  \geq\Omega\left(  n^{\log144/\log
24}\right)  \geq\Omega\left(  n^{1.563}\right)  , & S_{R(6)}^{\prime}\left(
n\right)  \geq\Omega\left(  n^{\log528/\log84}\right)  \geq\Omega\left(
n^{1.414}\right)  ,\\
S_{R(8)}^{\prime}\left(  n\right)  \geq\Omega\left(  n^{\log1312/\log
208}\right)  \geq\Omega\left(  n^{1.345}\right)  , & S_{R(10)}^{\prime}\left(
n\right)  \geq\Omega\left(  n^{\log2640/\log420}\right)  \geq\Omega\left(
n^{1.304}\right)  ,\text{ and}\\
S_{R(5)}^{\prime}\left(  n\right)  \geq\Omega\left(  n^{\log264/\log
120}\right)  \geq\Omega\left(  n^{1.519}\right)  . &
\end{array}
\]

\end{theorem}

In Section \ref{sec: Equi general m}, we briefly explore the behavior of the
function $S_{P}(n,m)$ for larger values of $m$ and $P=\triangle$ the
equilateral triangle. We then present general asymptotic results, for
arbitrary patterns, when $m=m(n)$ is a function of $n$ such that
~$m(n)\rightarrow\infty$ when $n\rightarrow\infty$. In this case we prove that
$S_{P}(n,m)\geq n^{2-\varepsilon}$ for every $n$ sufficiently large.

After having relative success constructing point sets with not many points on
a line, we impose harder restrictions by prohibiting parallelograms (and
collinear points) in the sets $Q$. This restriction immediately forbids the
use of Minkovski Sums. We are able to construct $n$-sets $Q$ with
$\Omega(n\log n)$ copies of a pattern $P$ without parallelograms (or collinear
triples). We also show a non-trivial upper bound for these patterns, namely,
we prove that at most $O(n^{3/2})$ copies of $P$ are possible. These results
are presented in Section \ref{sec:noparall}.

\section{Proof of Theorem \ref{th:main}\label{sec: Main}}

Let $P$ and $A$ be sets with no $m$ collinear points. With $A$ as a base set,
we recursively construct large sets with no $m$ collinear points and with
large number of similar copies of $P$. Our main tool is the \emph{Minkovski
Sum} of two sets $A,B\subseteq\mathbb{C}$, defined as the set $A+B=\left\{
a+b:a\in A,b\in B\right\}  $. First we have the following observation.

\begin{proposition}
Let $P=\left\{  p_{1},p_{2},\ldots,p_{k}\right\}  $ and $Q=\left\{
q_{1},q_{2},\ldots,q_{k}\right\}  $ be sets with $k$ elements. $P$ is similar
to $Q$, with $p_{j}$ corresponding to $q_{j}$ if and only if
\[
\frac{q_{j}-q_{1}}{q_{2}-q_{1}}=\frac{p_{j}-p_{1}}{p_{2}-p_{1}}\text{ for
}j=1,2,\ldots,k\text{.}%
\]

\end{proposition}

\begin{proof}
If $P\sim Q$ with $q_{j}=zp_{j}+w$, where $z\neq0$ and $w$ are fixed complex
numbers, then%
\[
\frac{q_{j}-q_{1}}{q_{2}-q_{1}}=\frac{\left(  zp_{j}+w\right)  -\left(
zp_{1}+w\right)  }{\left(  zp_{2}+w\right)  -\left(  zp_{1}+w\right)  }%
=\frac{p_{j}-p_{1}}{p_{2}-p_{1}}\text{.}%
\]
Reciprocally, if $\left(  q_{j}-q_{1}\right)  /\left(  q_{2}-q_{1}\right)
=\left(  p_{j}-p_{1}\right)  /\left(  p_{2}-p_{1}\right)  $ for $1\leq j\leq
k$, then letting $z=(q_{2}-q_{1})/(p_{2}-p_{1})$ and $w=q_{1}-zp_{1}$ we get
that $q_{j}=zp_{j}+w$ and $z\neq0$.
\end{proof}

We now bound the number of copies of $P$ in the sum $A+B$. To be concise, let
$I=\left\vert \mathrm{Iso}^{+}(P)\right\vert $.

\begin{lemma}
\label{lem:Minkovski}Let $P$ be any finite pattern, and $B$ and $C$ finite
sets such that $B+C$ has exactly $\left\vert B\right\vert \left\vert
C\right\vert $ points. Then%
\[
I\cdot S_{P}(B+C)+\left\vert B\right\vert \left\vert C\right\vert \geq\left(
I\cdot S_{P}(B)+\left\vert B\right\vert \right)  \left(  I\cdot S_{P}%
(C)+\left\vert C\right\vert \right)  \text{.}%
\]

\end{lemma}

\begin{proof}
Suppose that $P=\left\{  p_{1},p_{2},\ldots,p_{k}\right\}  $ and let
$\lambda_{j}$ denote the ratio $\left(  p_{j}-p_{1}\right)  /\left(
p_{2}-p_{1}\right)  $. Let $P_{B}=\left\{  b_{1},b_{2},\ldots,b_{k}\right\}
\subseteq B$ and $P_{C}=\left\{  c_{1},c_{2},\ldots,c_{k}\right\}  \subseteq
C$ be corresponding copies of $P$ with $P\sim P_{B}\sim P_{C}$. Then by the
previous proposition,
\[
\frac{b_{j}-b_{1}}{b_{2}-b_{1}}=\frac{c_{j}-c_{1}}{c_{2}-c_{1}}=\frac
{p_{j}-p_{1}}{p_{2}-p_{1}}=\lambda_{j}.
\]
Since $B+C$ has exactly $\left\vert B\right\vert \left\vert C\right\vert $
elements, then the relation $(b,c)\leftrightarrow b+c$ for $b\in B$, $c\in C$
is bijective. For any orientation-preserving isometry $f$ of $P_{C}$ (which
uniquely corresponds to an element of $\mathrm{Iso}^{+}(P)$), consider the set
$Q=\left\{  q_{j}:=b_{j}+f\left(  c_{j}\right)  :j=1,2,...,k\right\}
\subseteq B+C$. Since $\left(  f\left(  c_{j}\right)  -f\left(  c_{1}\right)
\right)  /\left(  f\left(  c_{2}\right)  -f\left(  c_{1}\right)  \right)
=\left(  c_{j}-c_{1}\right)  /\left(  c_{2}-c_{1}\right)  =\lambda_{j}$, then%
\[
\frac{q_{j}-q_{1}}{q_{2}-q_{1}}=\frac{b_{j}-b_{1}+f\left(  c_{j}\right)
-f\left(  c_{1}\right)  }{b_{2}-b_{1}+f\left(  c_{2}\right)  -f\left(
c_{1}\right)  }=\frac{\lambda_{j}\left(  b_{2}-b_{1}+f\left(  c_{2}\right)
-f\left(  c_{1}\right)  \right)  }{b_{2}-b_{1}+f\left(  c_{2}\right)
-f\left(  c_{1}\right)  }=\lambda_{j}.
\]
That is, by the previous proposition, $Q\sim P$. Thus every similar copy
$P_{B}$ of $P$ in $B$ together with a similar copy $P_{C}$ of $P$ in $C$
originate $I$ distinct similar copies of $P$ in $B+C$. We also have the
`liftings' of $P$ in $B$ and $C$. That is, similar copies of $P$ of the form
$(b,c_{1}),(b,c_{2}),\ldots,(b,c_{k})$ or $(b_{1},c),(b_{2},c),\ldots
,(b_{k},c)$, with $b\in B$ and $c\in C$. All these copies of $P$ in $B+C$ are
different because $\left\vert B+C\right\vert =\left\vert B\right\vert
\left\vert C\right\vert $. Therefore the number of similar copies of $P$ in
$B+C$ is at least $I\cdot S_{P}\left(  B\right)  S_{P}\left(  C\right)
+\left\vert B\right\vert S_{P}\left(  C\right)  +\left\vert C\right\vert
S_{P}\left(  B\right)  $. In other words%
\begin{align*}
I\cdot S_{P}\left(  B+C\right)  +\left\vert B\right\vert \left\vert
C\right\vert  &  \geq I^{2}\cdot S_{P}\left(  B\right)  S_{P}\left(  C\right)
+I\cdot\left\vert B\right\vert S_{P}\left(  C\right)  +I\cdot\left\vert
C\right\vert S_{P}\left(  B\right)  +\left\vert B\right\vert \left\vert
C\right\vert \\
&  =\left(  I\cdot S_{P}\left(  B\right)  +\left\vert B\right\vert \right)
\left(  I\cdot S_{P}\left(  C\right)  +\left\vert C\right\vert \right)
.\qedhere
\end{align*}

\end{proof}

The next lemma allows us to preserve the maximum number of collinear points
when we use the Minkovsky Sum of two appropriate sets.

\begin{lemma}
\label{lem:tech nomcoll}Let $A$ and $B$ be two sets with no $m$ points on a
line, $m\geq3$. If $\mathcal{S}$ is the set of points $v\in\mathbb{C}$ for
which $A+vB$ has less than $\left\vert A\right\vert \left\vert B\right\vert $
points, or $m$ points on a line; then $\mathcal{S}$ has zero Lebesgue measure.
\end{lemma}

For presentation purposes we defer its proof and instead proceed to the proof
of the theorem.

\begin{proof}
[Proof of Theorem \ref{th:main}]Let $A_{1}=A_{1}^{\ast}=A$ and suppose $A_{j}$
and $A_{j}^{\ast}$ have been defined. By Lemma \ref{lem:tech nomcoll} there is
a set $A_{j+1}$, similar to $A$, such that $A_{j+1}^{\ast}:=A_{j}^{\ast
}+A_{j+1}$ does not have $m$ points on a line and $|A_{j+1}^{\ast}%
|=|A_{j}^{\ast}||A_{j+1}|$. For every $j\geq1$, $|A_{j}^{\ast}|=|A_{j-1}%
^{\ast}||A|=|A_{j-2}^{\ast}||A|^{2}=\cdots=|A|^{j}$. Moreover, by Lemma
\ref{lem:Minkovski}, it follows that%
\begin{align*}
I\cdot S_{P}\left(  A_{j}^{\ast}\right)  +\left\vert A_{j}^{\ast}\right\vert
&  =I\cdot S_{P}\left(  A_{j-1}^{\ast}+A_{j}\right)  +\left\vert A_{j-1}%
^{\ast}\right\vert \left\vert A_{j}\right\vert \\
&  \geq\left(  I\cdot S_{P}\left(  A_{j-1}^{\ast}\right)  +\left\vert
A_{j-1}^{\ast}\right\vert \right)  \left(  I\cdot S_{P}\left(  A\right)
+\left\vert A\right\vert \right)  \\
&  \geq\cdots\\
&  \geq\left(  I\cdot S_{P}\left(  A\right)  +\left\vert A\right\vert \right)
^{j}\text{.}%
\end{align*}

If $S_{P}(A)=0$, then $i_{P}(A)=1$ and the result is trivial. Assume
$i_{P}(A)>1$ and suppose $\left\vert A\right\vert ^{j}\leq n<\left\vert
A\right\vert ^{j+1}$. The previous inequality yields%
\begin{align}
S_{P}(n,m)  &  \geq S_{P}(A_{j}^{\ast})\geq\frac{1}{I}(\left(  I\cdot
S_{P}(A)+\left\vert A\right\vert \right)  ^{j}-\left\vert A\right\vert
^{j})\label{eq: iteration}\\
&  \geq\frac{1}{I}\left(  \left\vert A\right\vert ^{j\cdot i_{P}(A)}-n\right)
\geq\frac{1}{I}\left(  \left(  \frac{n}{\left\vert A\right\vert }\right)
^{i_{P}(A)}-n\right) \label{eq: ineqmain}\\
&  \geq cn^{i_{P}(A)},\nonumber
\end{align}
for some constant $c=c(A,P)$ if $n$ is large enough. For instance,
$c=(2I\left\vert A\right\vert ^{i_{P}(A)})^{-1}$ works whenever $n^{i_{P}%
(A)-1}\geq2\left\vert A\right\vert ^{i_{P}(A)}$.
\end{proof}

Figure \ref{fig: minkovski} shows a set $A_{3}^{\ast}$ obtained from this
procedure when $P=\triangle$ the equilateral triangle and $A$ is the starting
set with 15 points and 29 equilateral triangles constructed in Section
\ref{sec: equilateral}. Finally, we present the proof of Lemma
\ref{lem:tech nomcoll}.

\begin{proof}
[Proof of Lemma \ref{lem:tech nomcoll}]We show that $\mathcal{S}$ is the union
of a finite number of algebraic sets, all of them of real dimension at most
one. This immediately implies that the Lebesgue measure of such a set is zero.
For every $a\in A$ and $b\in B,$ let $q(a,b)=a+vb$. Suppose that
$q(a_{1},b_{1})=q(a_{2},b_{2})$ with $(a_{1},b_{1})\neq(a_{2},b_{2})$. Then
$v(b_{2}-b_{1})+(a_{2}-a_{1})=0$ and $b_{1}\neq b_{2}$. Thus $v=-(a_{2}%
-a_{1})/(b_{2}-b_{1})$. Therefore there are at most $\tbinom{\left\vert
A\right\vert \left\vert B\right\vert }{2}$ values of $v$ for which $A+vB$ has
less than $\left\vert A\right\vert \left\vert B\right\vert $ points.

Now, suppose that the set $\{q(a_{j},b_{j}):1\leq j\leq m\}$ consists of $m$
points on a line, where $\{a_{j}\}\subseteq A$ and $\{b_{j}\}\subseteq B$.
Then for $3\leq j\leq m$, we have $q(a_{j},b_{j})-q(a_{1},b_{1})=\lambda
_{j}\left(  q(a_{2},b_{2})-q(a_{1},b_{1})\right)  $ where $\lambda_{j}\neq0,1$
is a real number. Thus for $3\leq j\leq m$,%
\begin{equation}
v\left(  b_{j}-b_{1}-\lambda_{j}\left(  b_{2}-b_{1}\right)  \right)
=\lambda_{j}\left(  a_{2}-a_{1}\right)  -\left(  a_{j}-a_{1}\right)  .
\label{eqn: v and lambda}%
\end{equation}
First assume all $b_{j}$ are equal. Then all $a_{j}$ are pairwise different,
otherwise we would have less than $m$ points initially. Moreover, Equation
(\ref{eqn: v and lambda}) implies that $(a_{j}-a_{1})/(a_{2}-a_{1}%
)=\lambda_{j}\in\mathbb{R}$ for all $3\leq j\leq m$. But this contradicts the
fact that there are no $m$ points on a line in $A$. By possibly relabeling the
points, we can now assume that $b_{1}\neq b_{2}$. If $b_{j}-b_{1}-\lambda
_{j}\left(  b_{2}-b_{1}\right)  \neq0$ for some $j$, then
\[
v=-\frac{\lambda_{j}\left(  a_{2}-a_{1}\right)  -\left(  a_{j}-a_{1}\right)
}{\lambda_{j}\left(  b_{2}-b_{1}\right)  -\left(  b_{j}-b_{1}\right)
}\text{.}%
\]
M\"{o}bius Transformations send circles (or lines) to circles (or lines); thus
the last equation, seen as a parametric equation on the real variable
$\lambda_{j}$, represents a circle (or a line) in the plane. The remaining
case is when
\begin{equation}
b_{j}-b_{1}-\lambda_{j}\left(  b_{2}-b_{1}\right)  =0\text{ for }3\leq j\leq
m. \label{eqn: zerocond}%
\end{equation}
Since $\lambda_{j}\neq0,1$, then $b_{j}\neq b_{1},b_{2}$. Suppose $b_{j}%
=b_{k}$ for $3\leq j<k\leq m$, then $\lambda_{j}=\lambda_{k}$ and by
(\ref{eqn: v and lambda}) and (\ref{eqn: zerocond}) we deduce that
$a_{j}=a_{k}$. This contradicts the fact that $q(a_{j},b_{j})$ and
$q(a_{k},b_{k})$ are two different points. Therefore all the $b_{j}$ are
pairwise different, and then by (\ref{eqn: zerocond}) all the $b_{j}$ are on a
line. This is a contradiction since there are no $m$ points on a line in $B$.

Therefore $\mathcal{S}$ is the union of a finite number of points and at most
$\tbinom{\left\vert A\right\vert \left\vert B\right\vert }{m}$ circles (or
lines), and consequently it has zero Lebesgue measure.
\end{proof}

\section{Constructions of the initial sets}

All the constructions of the initial sets $A$ we provide have
explicit coordinates so it is possible to calculate $S_{P}(A)$ via
the following algorithm \cite{Br02}: Fix two points $p_{1},p_{2}\in
P$, for every ordered pair $(a_{1},a_{2})\in A\times A$ of distinct
points consider the unique orientation-preserving similarity
transformation $f$ that maps $p_{1}\mapsto a_{1}$ and $p_{2}\mapsto
a_{2}$. An explicit expression is $f(z)=\frac
{a_{1}-a_{2}}{p_{1}-p_{2}}z+\frac{a_{2}p_{1}-a_{1}p_{2}}{p_{1}-p_{2}}$.
Then verify whether $f(P)\subseteq A$. If $N$ is the number of pairs
$(a_{1} ,a_{2})$ for which $f(P)\subseteq A$, then
$S_{P}(A)=N/\left\vert \mathrm{Iso}^{+}(P)\right\vert $. The running
time of this algorithm is $O(|P||A|^{2}\log|A|)$. Likewise, it is
possible to verify that no 3 points are on a line by simply checking
the pairwise slopes of every triple of distinct points. We first
present our construction for arbitrary patterns $P$.

\subsection{Arbitrary pattern $P$\label{sec: arbitrary}}

\begin{proof}
[Proof of Theorem \ref{th: general P}]Let $z_{0}\in\mathbb{C}\backslash P$ be
an arbitrary point and let $p_{1},p_{2}\in P$ be two fixed points in $P$. For
every $p\in P$, there is exactly one orientation-preserving similarity
function $f_{p}$ such that $p_{1}\mapsto z_{0}$ and $p_{2}\mapsto p$; indeed
an explicit expression of such function is $f_{p}(z)=\left(  \frac{z_{0}%
-p}{p_{1}-p_{2}}\right)  z+\frac{pp_{1}-z_{0}p_{2}}{p_{1}-p_{2}}$.
Let $A$ be the point set obtained by taking the image of $P$ under
every one of the functions $f_{p}$, that is $A=\bigcup_{p\in
P}f_{p}(P)$. For almost all $z_0$, except for a subset of real
dimension 1, the set $A$ does not have $m$ points on a line and all
the sets $f_{p}(P)\backslash\{z_{0}\}$ are pairwise disjoint, that
is $|A|=1+k\left(  k-1\right)  =k^{2}-k+1$. This fact can be proved
along the same lines as Lemma \ref{lem:tech nomcoll}, we omit the
details.
%TCIMACRO{\FRAME{fhFU}{6.0502in}{2.7691in}{0pt}{\Qcb{$A$ has no $m$ points on a
%line, $\left\vert A\right\vert =\left\vert P\right\vert ^{2}-\left\vert
%P\right\vert +1$ points and $S_{P}(A)=2\left\vert P\right\vert -1$. }}%
%{}{generalp.eps}{\special{ language "Scientific Word";  type "GRAPHIC";
%width 6.0502in;  height 2.7691in;  depth 0pt;  original-width 5.5287in;
%original-height 2.7in;  cropleft "0";  croptop "1";  cropright "1";
%cropbottom "0";  filename 'generalP.eps';file-properties "XNPEU";}}}%
%BeginExpansion
\begin{figure}
[h]
\begin{center}
\includegraphics[
height=2.7691in,
width=6.0502in
]%
{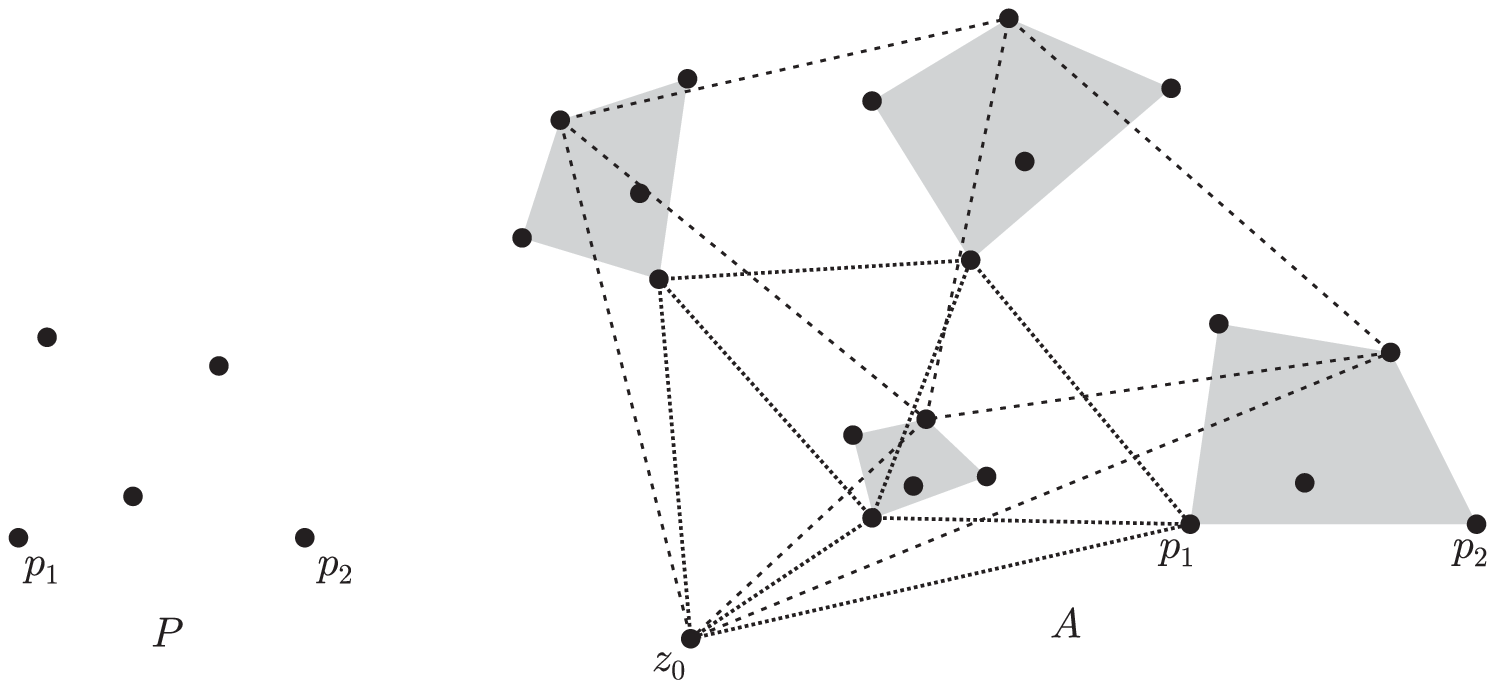}%
\caption{$A$ has no $m$ points on a line, $\left\vert A\right\vert =\left\vert
P\right\vert ^{2}-\left\vert P\right\vert +1$ points and $S_{P}(A)=2\left\vert
P\right\vert -1$. }%
\end{center}
\end{figure}
%EndExpansion
By construction, each of the $k$ sets $f_{p}(P)$ is similar to $P$. For every
$q\in P\backslash\{p_{1}\}$, the set%
\begin{equation}
\left\{  f_{p}(q):p\in P\right\}  =\frac{p_{1}-q}{p_{1}-p_{2}}P+\frac
{z_{0}\left(  q-p_{2}\right)  }{p_{1}-p_{2}}\label{eq: general P}%
\end{equation}
is also similar to $P$ and different from the previous copies of $P$ we
counted before. Thus $S_{P}(A)\geq2k-1$ and
\[
i_{P}(A)\geq\frac{\log\left(  S_{P}(A)+\left\vert A\right\vert \right)  }%
{\log\left\vert A\right\vert }=\frac{\log\left(  k^{2}+k\right)  }{\log\left(
k^{2}-k+1\right)  }\text{.}%
\]
The conclusion follows from Theorem \ref{th:main}.
\end{proof}

\subsection{Triangles}

The following table gives the currently best available initial set $A$ for
each pattern $P$ in Theorem \ref{th:initialsetstriang}. That is, the set $A$
with the largest index $i_{P}\left(  A\right)  $ known to date. The lower
bound stated in Theorem \ref{th:initialsetstriang} is then given by Theorem
\ref{th:main} applied to $A$.

\begin{table}[h]
\begin{center}%
\begin{tabular}
[c]{lc|ccc}
&  & \multicolumn{3}{c}{Best available $A$}\\
Triangular pattern $P=T$ & \multicolumn{1}{|c|}{$\left\vert \mathrm{Iso}%
^{+}(T)\right\vert $} & $\left\vert A\right\vert $ & $S_{T}\left(  A\right)  $
& $i_{T}\left(  A\right)  $\\\hline\hline
most scalene triangles & \multicolumn{1}{|c|}{$1$} & $14$ & $26$ &
\multicolumn{1}{l}{$\log40/\log14>1.\,\allowbreak397$}\\
all scalene triangles & $1$ & $5$ & $4$ & \multicolumn{1}{l}{$\log
9/\log5>1.\,\allowbreak365$}\\%
\begin{tabular}
[c]{l}%
isosceles triangle with\\
largest angle $\neq2\pi/3,\pi/2,\pi/3$%
\end{tabular}
& \multicolumn{1}{|c|}{$1$} & $8$ & $9$ & \multicolumn{1}{l}{$\log
17/\log8>\allowbreak1.\,\allowbreak362$}\\
$\left(  2\pi/3,\pi/6,\pi/6\right)  $-isosceles triangle &
\multicolumn{1}{|c|}{$1$} & $84$ & $444$ & \multicolumn{1}{l}{$\log
528/\log84>1.\,\allowbreak414$}\\
$\left(  \pi/2,\pi/4,\pi/4\right)  $-isosceles triangle &
\multicolumn{1}{|c|}{$1$} & $24$ & $120$ & \multicolumn{1}{l}{$\log
144/\log24>\allowbreak1.\,\allowbreak563$}\\
equilateral triangle & \multicolumn{1}{|c|}{$3$} & $15$ & $29$ &
\multicolumn{1}{l}{$\log102/\log15=\allowbreak1.\,\allowbreak707$}%
\end{tabular}
\end{center}
\caption{Indices for the best initial sets when $P=T$ is a triangle.}%
\end{table}

It is worth noting that our bound for scalene triangles is better than the one
for (most) isosceles triangles. The intuitive reason for this is that it is
harder to obtain better initial sets when the pattern has any type of
symmetries. This difficulty is overtaken by the factor $\left\vert
\mathrm{Iso}^{+}(P)\right\vert =3$ when $P$ is an equilateral triangle. We
extend our comments in the concluding remarks.%
%TCIMACRO{\FRAME{fhFU}{4.7089in}{2.4491in}{0pt}{\Qcb{(a) 5-point set $A_{1}$
%with $S_{T}(A_{1})=4$, (b) 14-point set $A_{2}$ with $S_{T}(A_{2})=26$. }%
%}{\Qlb{Fig: scalene}}{scalene.eps}{\special{ language "Scientific Word";
%type "GRAPHIC";  maintain-aspect-ratio TRUE;  display "USEDEF";
%valid_file "F";  width 4.7089in;  height 2.4491in;  depth 0pt;
%original-width 4.6112in;  original-height 2.4206in;  cropleft "0";
%filename 'scalene.eps';file-properties "XNPEU";}}}%
%BeginExpansion
\begin{figure}
[h]
\begin{center}
\includegraphics[
height=2.4491in,
width=4.7089in
]%
{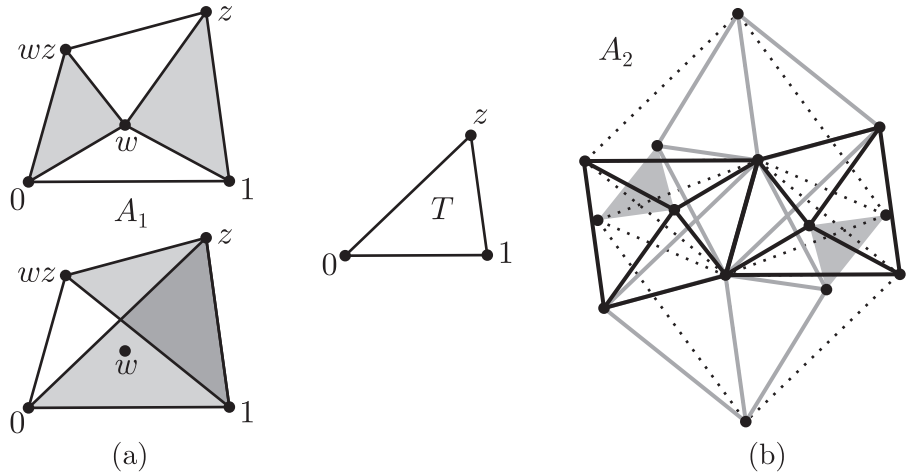}%
\caption{(a) 5-point set $A_{1}$ with $S_{T}(A_{1})=4$, (b) 14-point set
$A_{2}$ with $S_{T}(A_{2})=26$. }%
\label{Fig: scalene}%
\end{center}
\end{figure}
%EndExpansion

\subsubsection{Scalene triangles\label{sec: scalene}}

We first give a construction for all scalene triangles. If the pattern $P=T$
consists of the points (complex numbers) $0,1,$ and $z\notin\mathbb{R}$, then
the initial set $A_{1}=\left\{  0,1,z,w=z-1+1/z,wz\right\}  $ is in general
position for every scalene triangle $T$, see Figure \ref{Fig: scalene}(a).
$A_{1}$ has $4$ triangles similar to $T$: $\left(  0,1,z\right)  ,\left(
z,w,1\right)  ,\left(  1,z,wz\right)  $ and $\left(  0,w,wz\right)  $. For all
$z$, but a $1$-dimensional subset of $\mathbb{C}$, the construction in Figure
\ref{Fig: scalene}(b) is in general position and gives a better lower bound
for $S_{T}^{\prime}\left(  n\right)  $. The corresponding initial set is
$A_{2}=A\cup A^{\prime}$ where $A=A_{1}\cup zA_{1}\cup\left\{  wz^{2}/\left(
z-1\right)  \right\}  $ and $A^{\prime}$ is the $\pi$-rotation of $A$ about
$wz/2$, that is $A^{\prime}=-A+wz$. Because some points overlap, $A_{2}$ has
only $14$ points,%
\[
A_{2}=\left\{  0,1,z,w,wz,z^{2},wz^{2},wz^{2}/\left(  z-1\right)
,wz-1,wz-z,wz-w,1-z,wz\left(  1-z\right)  ,wz/\left(  1-z\right)  \right\}  .
\]
Among the points in $A_{2}$, there are $6$ similar copies of $A_{1}$ with no
triangles similar to $T$ in common: $A_{1}$, $zA_{1}$, $\frac{z}{1-z}\left(
A_{1}-wz\right)  $, and their $\pi$-rotations about $wz/2$. In addition, the
triangles $(w,wz(1-z),wz^{2}/(z-1))$ and $(wz^{2},wz-w,wz/(1-z))$ are similar
to $T$ and are not contained in the 6 sets similar to $A_{1}$ mentioned
before, so $S_{T}(A)\geq26$.%
%TCIMACRO{\FRAME{fhFU}{6.5017in}{2.3679in}{0pt}{\Qcb{Sets $A_{i}$ with
%$\left\vert A_{i}\right\vert =8$ and $S_{T(\beta)}(A_{i})=9$. $A_{1}$ is in
%general position for $\alpha\neq\pi/12,\pi/6,\pi/4,\pi/3,$ or $5\pi/12$,
%$A_{2}$ is in general position for $\alpha\neq\arccos\sqrt{(3+\sqrt{17})/8}$,
%$\pi/6$, $\pi/4$, or $\pi/3$.}}{\Qlb{Fig: isosceles}}{isosceles.eps}%
%width 6.5017in;  height 2.3679in;  depth 0pt;  original-width 5.8773in;
%original-height 2.2684in;  cropleft "0";  croptop "1";  cropright "1";
%cropbottom "0";  filename 'isosceles.eps';file-properties "XNPEU";}}}%
%BeginExpansion
\begin{figure}
[h]
\begin{center}
\includegraphics[
height=2.3679in,
width=6.5017in
]%
{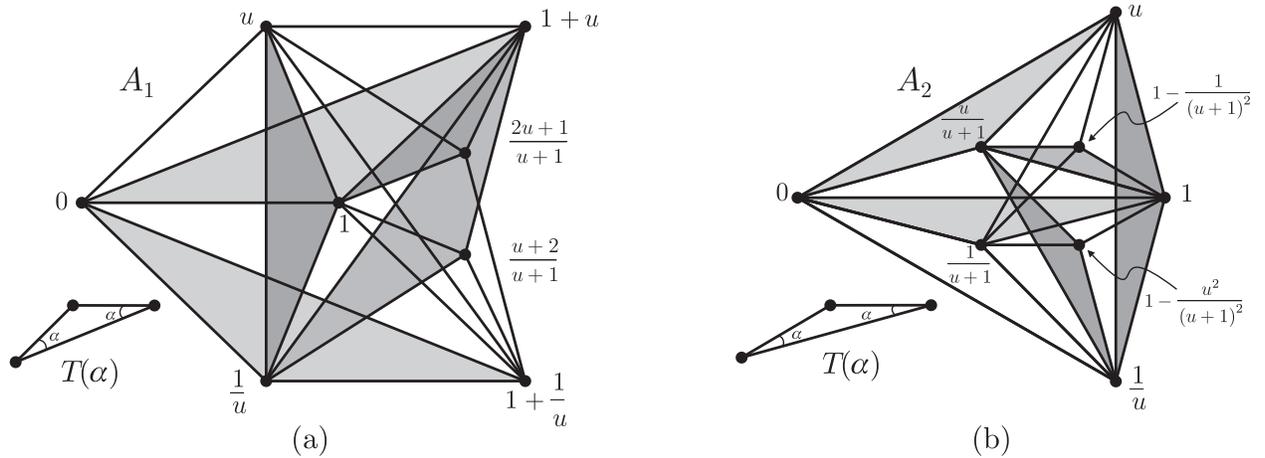}%
\caption{Sets $A_{i}$ with $\left\vert A_{i}\right\vert =8$ and $S_{T(\beta
)}(A_{i})=9$. $A_{1}$ is in general position for $\alpha\neq\pi/12,\pi
/6,\pi/4,\pi/3,$ or $5\pi/12$, $A_{2}$ is in general position for $\alpha
\neq\arccos\sqrt{(3+\sqrt{17})/8}$, $\pi/6$, $\pi/4$, or $\pi/3$.}%
\label{Fig: isosceles}%
\end{center}
\end{figure}
%EndExpansion

\subsubsection{Isosceles triangles\label{sec: isosceles}}

Let the pattern $P=T(\alpha)$ be an isosceles triangle with angles
$\alpha,\alpha$, and $\pi-2\alpha$, where $0<\alpha<\pi/2$. We use as initial
set one of the following two constructions, each with $8$ points and $9$
copies of $T(\alpha)$, i.e., $S_{T(\alpha)}(A)=9$. There are three exceptions
that are analyzed later in Section \ref{sec: subsets of reg}: $T(\pi/6),$
$T(\pi/4)$, and the equilateral triangle $T(\pi/3)$. Let $u=e^{2\alpha i}$ so
that $T(\alpha)=\{0,1,-u\}$. The first initial set is $A_{1}=B_{1}%
\cup\overline{B_{1}}$ where $B_{1}=\left\{  0,1,u,1+u,\frac{2u+1}%
{u+1}\right\}  $ and $\overline{B_{1}}$ is the conjugate of $B_{1}$, i.e.,
$B_{1}=\{\overline{b}:b\in B_{1}\}$. (See Figure \ref{Fig: isosceles}(a).)
This configuration is in general position as long as $\alpha\neq k\pi/12$,
$k\in\mathbb{Z}$. It has $9$ copies of $T(\alpha)$: $\left(  0,1,1+u\right)
$, $\left(  1+u,u,0\right)  $, $\left(  1,\frac{2u+1}{u+1},1+u\right)  $,
$\left(  1+1/u,\frac{2u+1}{u+1},u\right)  $, their reflections about the real
axis, and $\left(  1/u,1,u\right)  $. The actual points in $A_{1}$ are%
\[
A_{1}=\left\{  0,1,u,\frac{1}{u},1+u,1+\frac{1}{u},\frac{u+2}{u+1},\frac
{2u+1}{u+1}\right\}  .
\]

The second initial set is $A_{2}=B_{2}\cup\overline{B_{2}}$ where
$B_{2}=\left\{  0,1,u,\frac{u}{u+1},1-\frac{1}{\left(  u+1\right)  ^{2}%
}\right\}  $. (See Figure \ref{Fig: isosceles}(b).) This set is in general
position as long as $\alpha\neq\arccos\sqrt{(3+\sqrt{17})/8},\pi/6,\pi/4,$ or
$\pi/3$. It has $9$ copies of $T(\alpha)$: $\left(  1,\frac{u}{1+1},0\right)
$, $\left(  0,\frac{u}{u+1},u\right)  $, $\left(  1,1-\frac{1}{\left(
u+1\right)  ^{2}},\frac{u}{u+1}\right)  $, $\left(  \frac{1}{1+u},1-\frac
{1}{\left(  u+1\right)  ^{2}},u\right)  $, their reflections about the real
axis, and $(1/u,1,u)$. The actual points in $A_{2}$ are%
\[
A_{2}=\left\{  0,1,u,\frac{1}{u},\frac{u}{u+1},\frac{1}{u+1},1-\frac
{1}{\left(  u+1\right)  ^{2}},1-\frac{u^{2}}{\left(  u+1\right)  ^{2}%
}\right\}  .
\]

\subsubsection{Equilateral triangle\label{sec: equilateral}}

Let $z\in\mathbb{C}$ and $\omega=e^{2i\pi/3}$ so that $\omega^{2}+\omega+1=0$.
When the pattern $P=\triangle=\{1,\omega,\omega^{2}\}$ is the equilateral
triangle, we use as initial set $A=B\cup\omega B\cup\omega^{2}B$ where
$B=\{1,-z\}\cup(-1+zP)$. For all $z$ but a $1$-dimensional subset of
$\mathbb{C}$, the set $A$ is in general position and $\left\vert A\right\vert
=15$. The set $B_{1}=\bigcup_{k=0}^{2}\omega^{k}(-1+zP)$ is the Minkovski Sum
$-P+zP$, thus by Lemma \ref{lem:Minkovski} there are at least 9 equilateral
triangles in $B_{1}$. In addition $(1,\omega,\omega^{2})$ and $(-z,-z\omega
,-z\omega^{2})$ are equilateral, and each of the points $1$, $\omega$, and
$\omega^{2}$ is incident to 6 more equilateral triangles: $(1,-\omega
^{2}+z,-\omega^{2}z)$, $(1,-\omega^{2}+z\omega,-z)$, $(1,-\omega^{2}%
+z\omega^{2},-z\omega)$, $(1,-\omega^{2}z,-\omega+z\omega)$, $(1,-z,-\omega
+z\omega^{2})$, and $(1,-\omega z,-\omega+z)$, together with the $\pi/3$- and
$2\pi/3$-rotations of these triangles about the origin. Thus $S_{\triangle
}(A)\geq29$. It can be checked that there are only 29 equilateral triangles in
$A$.%
%TCIMACRO{\FRAME{fhFU}{6.5017in}{1.8049in}{0pt}{\Qcb{Initial set $A$ with
%$\left\vert A\right\vert =15$ and $S_{\triangle}(A)=29$.}}{}{equi.eps}%
%width 6.5017in;  height 1.8049in;  depth 0pt;  original-width 5.5711in;
%original-height 1.5714in;  cropleft "0";  croptop "1";  cropright "1";
%cropbottom "0";  filename 'equi.eps';file-properties "XNPEU";}}}%
%BeginExpansion
\begin{figure}
[h]
\begin{center}
\includegraphics[
height=1.8049in,
width=6.5017in
]%
{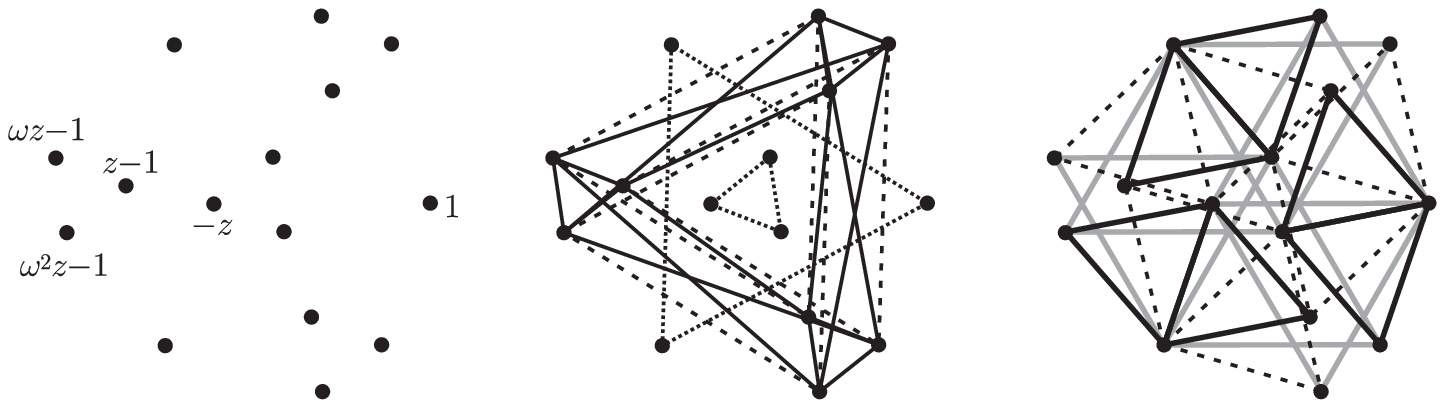}%
\caption{Initial set $A$ with $\left\vert A\right\vert =15$ and $S_{\triangle
}(A)=29$.}%
\end{center}
\end{figure}
%EndExpansion

\subsection{Regular polygons\label{sec: Regular Poly}}

As in the previous section, we construct initial sets for each $k$-regular
polygon with $k\in\{4$, $5$, $6$, $8$, $10\}$. The following table gives the
currently best available initial set $A$ for each of these regular polygons.

\begin{table}[h]
\begin{center}%
\begin{tabular}
[c]{lc|ccc}
&  & \multicolumn{3}{c}{Best available $A$}\\
Regular polygon $R(k)$ & \multicolumn{1}{|c|}{$\left\vert \mathrm{Iso}%
^{+}(R(k))\right\vert $} & $\left\vert A\right\vert $ & $S_{R(k)}\left(
A\right)  $ & $i_{R(k)}\left(  A\right)  $\\\hline\hline
Square $=R(4)$ & \multicolumn{1}{|c|}{$4$} & $24$ & $30$ &
\multicolumn{1}{l}{$\log144/\log24>1.\,\allowbreak563$}\\
Hexagon $=R(6)$ & \multicolumn{1}{|c|}{$6$} & $74$ & $84$ &
\multicolumn{1}{l}{$\log528/\log84>\allowbreak1.\,\allowbreak414$}\\
Octagon $=R(8)$ & \multicolumn{1}{|c|}{$8$} & $208$ & $138$ &
\multicolumn{1}{l}{$\log1312/\log208>1.\,\allowbreak345$}\\
Decagon $=R(10)$ & \multicolumn{1}{|c|}{$10$} & $420$ & $222$ & $\log
2640/\log420>\allowbreak1.\,\allowbreak304$\\
Pentagon $=R(5)$ & \multicolumn{1}{|c|}{$5$} & $120$ & $264$ &
\multicolumn{1}{l}{$\log1440/\log120>\allowbreak1.\,\allowbreak519$}%
\end{tabular}
\end{center}
\caption{Indices for the best initial sets for regular polygons.}%
\end{table}

We first present the construction for the even-sided polygons and then the
construction for the regular pentagon. Finally we explain how to use these
initial sets for any pattern that is a subset of a regular polygon.

\subsubsection{Even sided regular polygons.\label{sec: even regular}}

The following construction of the initial set $A$ is in general position for
all even $k$, however the index $i_{R(k)}(A)$ is only better than
$i_{R(k)}(R(k))$ when $k\leq10$. Let $\omega=e^{2\pi i/k}$, $k$ even, and
$z\in\mathbb{C}$ an arbitrary nonzero complex number. Suppose that the regular
$k$-gon $R(k)$ is given by $R(k)=P=1+\omega+z\{\omega^{j}:0\leq j\leq k-1\}$.
The reason why we translated the canonical regular polygon by $1+\omega$ and
rotated and magnified it by $z$ will become apparent soon. To construct our
initial set $A$, we first follow the construction of Theorem
\ref{th: general P} applied to $P$ with $p_{1}=1+\omega+z$, $p_{2}%
=1+\omega+z\omega$, and $z_{0}=2$. We obtain a set $A_{1}$ consisting of
$z_{0}$ and $k-1$ disjoint similar copies of $P$ given by (\ref{eq: general P}%
). That is, $A_{1}=\{2\}\cup\bigcup_{j=1}^{k-1}B_{j}$ where%
\[
B_{j}=\frac{1-\omega^{j}}{1-\omega}P+\frac{2\left(  \omega^{j}-\omega\right)
}{1-\omega}\text{, }1\leq j\leq k-1\text{.}%
\]
Note that there are exactly $k$ similar copies of $P$ with vertex $z_{0}=2$.
Furthermore, $\omega^{k/2}=-1$ because $k$ is even and thanks to the
translation by $1+\omega$ in the definition of $P$, we have that
\[
B_{k/2}=\left(  \frac{2z}{1-\omega}\right)  \left\{  \omega^{j}:0\leq j\leq
k-1\right\}  ,
\]
that is $B_{k/2}$ is a $k$-regular polygon centered at the origin. Now we add
to the construction every rotation of $A_{1}$ by an integer multiple of
$2\pi/k$. More precisely, we let%
\begin{equation}
A=\bigcup_{l=0}^{k-1}\omega^{l}A_{1}\text{.}\label{eq: union}%
\end{equation}%
%TCIMACRO{\FRAME{fhFU}{6.5017in}{1.8896in}{0pt}{\Qcb{Initial set $A$ for
%$R(4)$. $\left\vert A\right\vert =24$ and $S_{P}(A)=30$.}}{}{square.eps}%
%width 6.5017in;  height 1.8896in;  depth 0pt;  original-width 6.7463in;
%original-height 1.9319in;  cropleft "0";  croptop "1";  cropright "1";
%cropbottom "0";  filename 'square.eps';file-properties "XNPEU";}}}%
%BeginExpansion
\begin{figure}
[h]
\begin{center}
\includegraphics[
height=1.8896in,
width=6.5017in
]%
{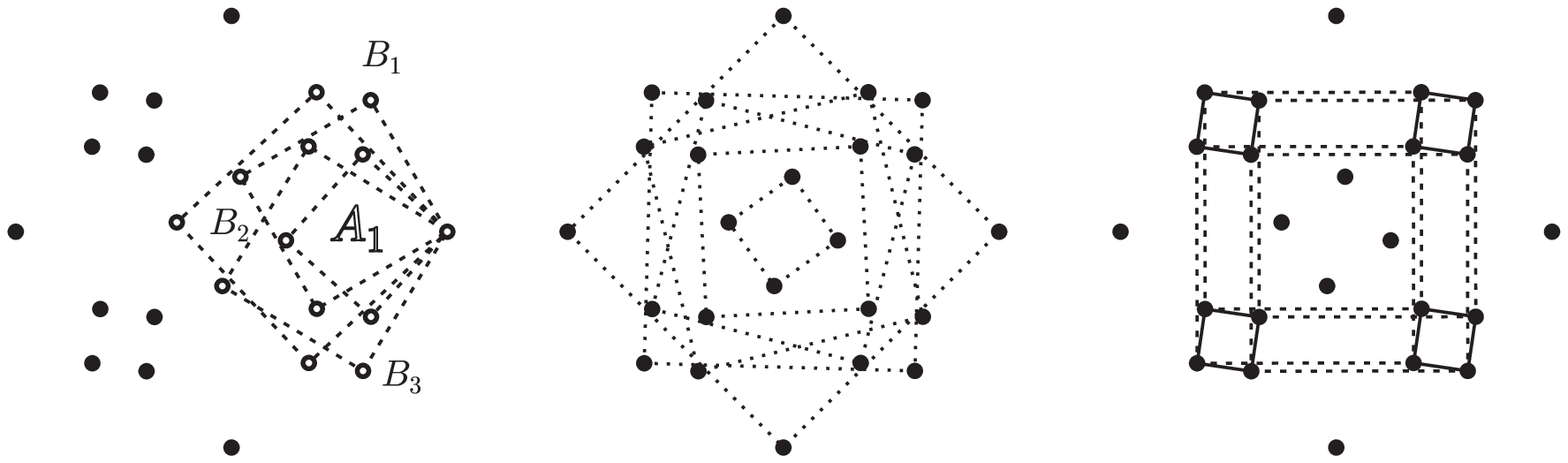}%
\caption{Initial set $A$ for $R(4)$. $\left\vert A\right\vert =24$ and
$S_{P}(A)=30$.}%
\end{center}
\end{figure}
%EndExpansion
Note that $B_{k/2}$ is a subset of all the sets $\omega^{l}A_{1}$. Because $k$
is even, $-P=P-2-2\omega$, thus for every $1\leq j\leq k/2-1$,
\[
\omega^{j}B_{k-j}=\frac{1-\omega^{j}}{1-\omega}\left(  -P\right)
+\frac{2\left(  1-\omega^{j+1}\right)  }{1-\omega}=\frac{1-\omega^{j}%
}{1-\omega}P+\frac{2\left(  \omega^{j}-\omega\right)  }{1-\omega}%
=B_{j}\text{.}%
\]
For almost all $z\in\mathbb{C}$, except for a subset of real dimension 1,
there are no 3 collinear points in $A$ and also for every $j\neq k/2$ \ the
sets $B_{j}$ and $\omega^{l}B_{i}$ are disjoint except for the pairs
$(i,l)=(j,0)$ and $(i,l)=(k-j,j)$. It follows that every set of the form
$\omega^{l}B_{j}$ with $j\neq k/2$ is a subset of exactly two terms in the
union from Equation (\ref{eq: union}) and it is disjoint from the rest. Thus%
\[
\left\vert A\right\vert =\left\vert B_{k/2}\right\vert +\left\vert
\bigcup_{\substack{0\leq l\leq k-1\\j\neq k/2}}\omega^{l}B_{j}\right\vert
+\left\vert \left\{  \omega^{l}z_{0}:0\leq l\leq k-1\right\}  \right\vert
=k+\frac{1}{2}k\left(  k-2\right)  +k=\frac{k}{2}\left(  k^{2}-2k+4\right)
\text{.}%
\]
To bound the number of $k$-regular polygons in $A$, first note that for each
$0\leq l\leq k-1$, there are at least $k$ regular polygons with a vertex in
$\omega^{l}z_{0}$ contained in $\omega^{l}A_{1}$ and all of these $k^{2}$
copies of $P$ are different. For every $1\leq j\leq k/2-1$ the set
$\bigcup_{l=0}^{k-1}\omega^{l}B_{j}$ is the Minkovski Sum of two copies of
$P$, namely
\[
P_{1}=\left(  1+\omega^{j}\right)  \left\{  \omega^{l}:0\leq l\leq
k-1\right\}  \text{ and }P_{2}=\frac{1-\omega^{j}}{1-\omega}z\left\{
\omega^{l}:0\leq l\leq k-1\right\}  \text{,}%
\]
with exactly $k^{2}$ points. Thus, by Lemma \ref{lem:Minkovski}, we have that
$S_{P}(\bigcup_{l=0}^{k-1}\omega^{l}B_{j})=S_{P}(P_{1}+P_{2})\geq3k$.
Furthermore, all these $3k(k/2-1)$ regular polygons are distinct and also
different from those previously counted. Finally, there are two extra polygons
not yet counted, namely $B_{k/2}$ and $\left\{  \omega^{l}z_{0}:0\leq l\leq
k-1\right\}  $. Thus $S_{P}(A)\geq k^{2}+3k(k/2-1)+2=\frac{1}{2}(5k^{2}-6k+4)$
and then%
\[
i_{P}(A)\geq\frac{\log\left(  \frac{k}{2}(5k^{2}-6k+4)+\frac{k}{2}\left(
k^{2}-2k+4\right)  \right)  }{\log\left(  \frac{k}{2}\left(  k^{2}%
-2k+4\right)  \right)  }=\frac{\log\left(  3k^{3}-4k^{2}+4k\right)  }%
{\log\left(  \frac{k}{2}\left(  k^{2}-2k+4\right)  \right)  }.
\]
The conclusion follows by setting $k=4,6,8,\,\ $or $10$.
%TCIMACRO{\FRAME{fhFU}{5.7216in}{2.4509in}{0pt}{\Qcb{Best known initial sets
%for $P=R(6)$ and $P=R(8)$.}}{}{hexa_and_octa.eps}%
%width 5.7216in;  height 2.4509in;  depth 0pt;  original-width 4.3038in;
%original-height 1.8165in;  cropleft "0";  croptop "1";  cropright "1";
%cropbottom "0";  filename 'hexa_and_octa.eps';file-properties "XNPEU";}}}%
%BeginExpansion
\begin{figure}
[hh]
\begin{center}
\includegraphics[
height=2.4509in,
width=5.7216in
]%
{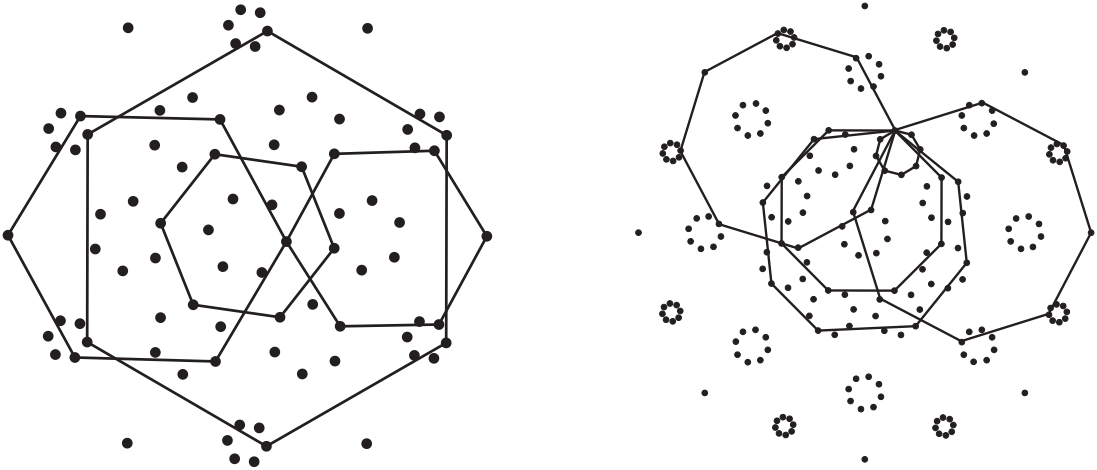}%
\caption{Best known initial sets for $P=R(6)$ and $P=R(8)$.}%
\end{center}
\end{figure}
%EndExpansion

\subsubsection{The regular pentagon}

Let $\omega=e^{2\pi i/5}$, for every $z\in\mathbb{C}$ set $R(5)=P=z\{1,\omega
,\omega^{2},\omega^{3},\omega^{4}\}$. Define
\[
A_{1}=P+\frac{\sqrt{5}+3}{2}\text{, }A_{2}=\frac{\sqrt{5}+1}{2}\left(
-P+1\right)  \text{, and }A_{3}=\left(  \omega^{2}-1\right)  \left\{
z,\omega+1\right\}  \text{.}%
\]
Now we consider all the $2\pi k/5$ rotations of these points, $0\leq k\leq4$,
as well as their symmetrical points with respect to the origin. That is, we
define $B=A_{1}\cup(-A_{1})\cup A_{2}\cup(-A_{2})\cup A_{3}\cup(-A_{3})$ and
$A=\bigcup_{k=0}^{4}\omega^{k}B$. For almost all $z\in\mathbb{C}$, except for
a subset of real dimension one, $A$ has exactly 120 points and has no three
points on a line. There are at least 264 regular pentagons with vertices in
$A$: for each $j=1,2$, the set $\bigcup_{k=0}^{4}\omega^{k}\left(  \pm
A_{j}\right)  $ is the Minkovski Sum of two regular pentagons, and thus by
Lemma \ref{lem:Minkovski} each of these 4 sets has 15 regular pentagons, the
point set $\bigcup_{k=0}^{4}\omega^{k}\left(  \pm A_{3}\right)  $ consists of
two regular decagons so it contains 4 pentagons, finally each of the 20 points
in $\bigcup_{k=0}^{4}\omega^{k}\left(  \pm A_{3}\right)  $ is incident to 10
more regular pentagons different from the ones previously counted (see Figure
\ref{fig: Penta}). In fact every point in $A$ is incident to exactly 11
regular pentagons and it turns out that the set $A$ has an interesting set of
automorphisms that preserve the regular pentagons.
%TCIMACRO{\FRAME{fhFU}{6.5017in}{2.0072in}{0pt}{\Qcb{The best initial set $A$
%for the regular pentagon}}{\Qlb{fig: Penta}}{penta.eps}%
%width 6.5017in;  height 2.0072in;  depth 0pt;  original-width 6.0995in;
%original-height 1.8645in;  cropleft "0";  croptop "1";  cropright "1";
%cropbottom "0";  filename 'penta.eps';file-properties "XNPEU";}}}%
%BeginExpansion
\begin{figure}
[h]
\begin{center}
\includegraphics[
height=2.0072in,
width=6.5017in
]%
{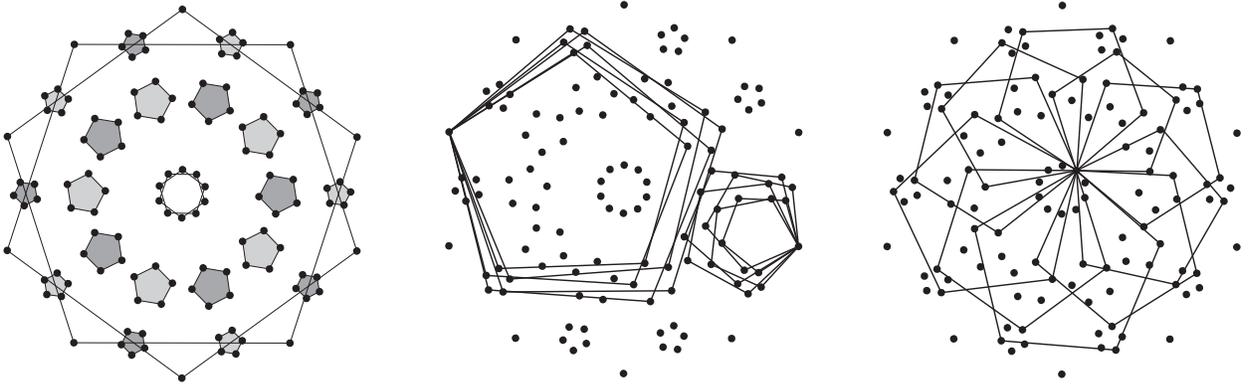}%
\caption{The best initial set $A$ for the regular pentagon}%
\label{fig: Penta}%
\end{center}
\end{figure}
%EndExpansion

\subsubsection{Subsets of regular polygons\label{sec: subsets of reg}}

If $P$ is a subset of a regular polygon $R$, then the constructions we have
previously obtained for $R$ would be also suitable for $P$. More precisely we
have the following result.

\begin{theorem}
Let $R$ be a regular polygon and $P\subseteq R$ with $|P|\geq3$. For every
nonempty finite $A\subseteq\mathbb{C}$, we have that%
\[
S_{P}(n)\geq\Omega\left(  n^{i_{R}(A)}\right)  \text{.}%
\]

\end{theorem}

\begin{proof}
Let $I=|\mathrm{Iso}^{+}(P)|$. Because $R$ is a regular polygon and $|P|\geq
3$, each similar copy of $P$ in $A$ is contained in at most one copy of $R$ in
$A$. Thus $S_{P}(A)\geq S_{R}(A)\cdot S_{P}(R)$. On the other hand,
$S_{P}(R)=\left\vert R\right\vert /I$ and thus%
\[
I\cdot S_{P}(A)\geq I\cdot S_{R}(A)\cdot S_{P}(R)=\left\vert R\right\vert
S_{R}(A).
\]
Consequently%
\[
i_{P}(A)=\frac{\log\left(  I\cdot S_{P}(A)+\left\vert A\right\vert \right)
}{\log\left\vert A\right\vert }\geq\frac{\log\left(  \left\vert R\right\vert
S_{R}(A)+\left\vert A\right\vert \right)  }{\log\left\vert A\right\vert
}=i_{R}(A)\text{.}%
\]
Finally, by Theorem \ref{th:main}, $S_{P}(n)\geq\Omega\left(  n^{i_{P}%
(A)}\right)  \geq\Omega\left(  n^{i_{R}(A)}\right)  $.
\end{proof}

As a direct consequence of this theorem, we take care of the isosceles
triangles for which the construction in Section \ref{sec: isosceles} yielded
collinear triples. For the isosceles triangles $T(\alpha)$ with $\alpha=\pi/6$
or $\pi/4$ we have that $S_{T(\pi/6)}(n)\geq\Omega\left(  n^{\log528/\log
84}\right)  $ and $S_{T(\pi/4)}(n)\geq\Omega(n^{\log144/\log24})$, both of
which exceed the bound in Theorem \ref{th:initialsetstriang}. Other point sets
treated before can be improved this way as well. For instance $S_{T(\pi
/5)}(n),S_{T(2\pi/5)}(n)\geq\Omega\left(  n^{\log1440/\log120}\right)
\geq\Omega(n^{1.519})$.

\subsection{$S_{P}\left(  n,m\right)  $ for an equilateral triangle
$P$\label{sec: Equi general m}}

When $m\geq4$, the initial sets $A_{m}$ with the largest indices we know are
clusters of points of the equilateral triangle lattice in the shape of a
circular disk.
%TCIMACRO{\FRAME{fhFU}{4.8646in}{2.7562in}{0pt}{\Qcb{Best known constructions
%of initial sets $A_{m}$ with many equilateral triangles and at most $m-1$
%points on a line. }}{\Qlb{Fig Hex}}{equitmany.eps}%
%width 4.8646in;  height 2.7562in;  depth 0pt;  original-width 4.8075in;
%original-height 3.717in;  cropleft "0";  croptop "1";  cropright "1";
%cropbottom "0";  filename 'equiTmany.eps';file-properties "XNPEU";}}}%
%BeginExpansion
\begin{figure}
[h]
\begin{center}
\includegraphics[
height=2.7562in,
width=4.8646in
]%
{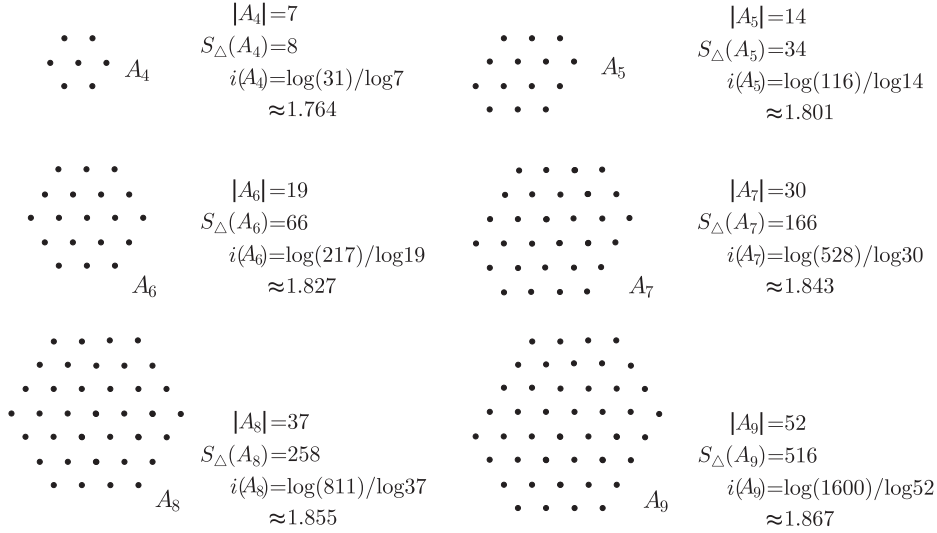}%
\caption{Best known constructions of initial sets $A_{m}$ with many
equilateral triangles and at most $m-1$ points on a line. }%
\label{Fig Hex}%
\end{center}
\end{figure}
%EndExpansion

\begin{theorem}
For $4\leq m\leq9$ and $P=\triangle$ the equilateral triangle, we have%
\[%
\begin{array}
[c]{ll}%
S_{\triangle}\left(  n,4\right)  \geq\Omega(n^{\log31/\log7})\geq\Omega\left(
n^{1.764}\right)  \hspace{0.4in} & S_{\triangle}\left(  n,5\right)  \geq
\Omega(n^{\log116/\log14})\geq\Omega\left(  n^{1.801}\right) \\
S_{\triangle}\left(  n,6\right)  \geq\Omega(n^{\log217/\log19})\geq
\Omega(n^{1.827}) & S_{\triangle}\left(  n,7\right)  \geq\Omega(n^{\log
528/\log30})\geq\Omega\left(  n^{1.843}\right) \\
S_{\triangle}\left(  n,8\right)  \geq\Omega(n^{\log811/\log37})\geq
\Omega\left(  n^{1.855}\right)  & S_{\triangle}\left(  n,9\right)  \geq
\Omega(n^{\log1600/\log52})\geq\Omega\left(  n^{1.867}\right)
\end{array}
\]
and in general if $m$ is even then
\[
S_{\triangle}\left(  n,m\right)  \geq\Omega\left(  n^{i_{\triangle}(A_{m}%
)}\right)
\]
where
\[
i_{\triangle}(A_{m})=\frac{\log(21m^{4}-84m^{3}+156m^{2}-144m+64)-\log64}%
{\log(3m^{2}-6m+4)-\log4}%
\]

\end{theorem}

\begin{proof}
For the first part refer to Figure \ref{Fig Hex} where the sets $A_{m}$ and
their corresponding indices are shown. For the second part we consider as our
set $A_{m}$ the lattice points inside a regular hexagon of side $m/2-1$ with
sides parallel to the lattice. Clearly $A_{m}$ contains at most $m-1$
collinear points. Also $\left\vert A_{m}\right\vert =\left(  3m^{2}%
-6m+4\right)  /4$ and $S_{\triangle}(A_{m})=\left(  7m^{4}-28m^{3}%
+36m^{2}-16m\right)  /64$ (see \cite{Abr97}), therefore the result follows
from Theorem \ref{th:main}.
\end{proof}

After some elementary estimations, we get that the index on the last theorem
satisfies that%
\[
i_{\triangle}(A_{m})\geq2-\frac{\log(12/7)}{2\log m}+\Theta\left(  \log
m\right)  ^{-2}>2-\frac{0.269}{\log m}+\Theta\left(  \log m\right)
^{-2}\text{.}%
\]
This suggests that $S_{\triangle}(n,m)\geq\Omega\left(  n^{2-0.269/\log
m}\right)  $, however the constant term hidden in the $\Omega$ may depend on
$A$, and thus on $m$. We see this with more detail on the next section.

\section{When $m$ grows together with $n$\label{sec: when m grows}}

We investigate the function $S_{P}(n,m)$ when the pattern $P$ is fixed and
$m=m(n)\rightarrow\infty$ when $n\rightarrow\infty$. For instance, the maximum
number of squares in a $n$-point set without $\log n$ points on a line, is at
least $\Omega(n^{2-c\left(  \log\log n\right)  ^{-1}})$. For the proof of our
result, we use the following two theorems mentioned in the introduction as the
best bounds for the function $S_{P}(n)$ without restrictions.

\begin{theoremA}
[Elekes and Erd\H{o}s \cite{EE94}]\label{th:ElekesErdos} For any pattern $P$
there are constants $a,b,c>0$ such that
\[
S_{P}(n)\geq cn^{2-a\left(  \log n\right)  ^{-b}},
\]
moreover, if the coordinates of $P$ are algebraic or if $\left\vert
P\right\vert =3$, then $S_{P}(n)\geq cn^{2}$.
\end{theoremA}

If $u,v,w,z\in\mathbb{C}$ then the \emph{cross-ratio} of the $4$-tuple
$(u,v,w,z)$ is defined as
\[
\frac{(w-u)(z-v)}{\left(  z-u\right)  \left(  w-v\right)  }.
\]

\begin{theoremA}
[Laczkovich and Ruzsa \cite{LR97}]\label{th:LaczRuzsa} $S_{P}(n)=\Theta
(n^{2})$ if and only if the cross-ratio of every $4$-tuple in $P$ is algebraic.
\end{theoremA}

For the sake of clarity, let us call a pattern $P$ \emph{cross-algebraic} if
the cross-ratio of every $4$-tuple is algebraic, and
\emph{cross-transcendental} otherwise.

\begin{theorem}
\label{th: m grows}Let $P$ be an arbitrary pattern and suppose
$m=m(n)\rightarrow\infty$, then for every $\varepsilon>0$ there is a threshold
function $N_{0}=N_{0}(\varepsilon,P)$ such that
\[
S_{P}(n,m)\geq n^{2-\varepsilon}\text{ for every }n\geq N_{0}\text{.}%
\]

\end{theorem}

\begin{proof}
Suppose $m=m(n)\rightarrow\infty$. We actually prove the following stronger result.

\begin{itemize}
\item[(i)] If $\log\left(  m\right)  \leq\sqrt{\log n}$ and $P$ is
cross-algebraic, then there is a constant $c_{1}>0$ depending only on $P$ such
that
\[
S_{P}(n,m)\geq\Omega\left(  n^{2-c_{1}/\log m}\right)  .
\]

\item[(ii)] If $\log\left(  m\right)  \leq\sqrt{\log n}$ and $P$ is
cross-transcendental, then there are constants $c_{1},c_{2},c_{3}>0$ depending
only on $P$ such that
\[
S_{P}(n,m)\geq\Omega\left(  n^{2-c_{2}/\left(  \log m\right)  ^{c_{3}}%
-c_{1}/\log m}\right)  .
\]

\item[(iii)] If $\log\left(  m\right)  >\sqrt{\log n}$, then $S_{P}(n,m)\geq
S_{P}(n,e^{\sqrt{\log n}})$ and thus either (i) or (ii) holds with $\log
m=\sqrt{\log n}$.
\end{itemize}

We first prove (i). Suppose that $P$ is cross-algebraic. By Theorem
\ref{th:LaczRuzsa} there is a constant $c$, depending only on $P$, and a
$(\left\lceil m\right\rceil -1)$-set $A$ such that $\left\vert A\right\vert
+S_{P}(A)\geq cm^{2}$. Clearly $A$ does not have $m$ points on a line. Then,
by (\ref{eq: ineqmain}) in the proof of Theorem \ref{th:main}, we have that%
\[
S_{P}(n,m)\geq\frac{1}{I}\left(  \left(  \frac{n}{\left\vert A\right\vert
}\right)  ^{i_{P}(A)}-n\right)  >\frac{1}{I}\left(  \left(  \frac{n}%
{m}\right)  ^{\frac{\log\left(  cm^{2}\right)  }{\log m}}-n\right)  =\frac
{1}{Ic}\left(  n^{2+\frac{\log c}{\log m}-\frac{2\log m}{\log n}}-cn\right)
\text{.}%
\]
By assumption, $\left(  2\log m\right)  /\log n\leq2/\log m$. Since $c<1$ it
follows that $\log c<0$. Let $c_{1}=2-\log c>0$, then
\[
S_{P}(n,m)\geq\frac{1}{Ic}\left(  n^{2+\frac{\log c}{\log m}-\frac{2}{\log m}%
}-n\right)  \geq\frac{1}{Ic}\left(  n^{2-c_{1}/\log m}-n\right)  .
\]
That is, $S_{P}(n,m)\geq\Omega\left(  n^{2-c_{1}/\log m}\right)  $, where the
constant in the $\Omega$ term does not depend on $n$ or $m$.

Similarly, to prove (ii), assume $P$ is cross-transcendental, then by Theorem
\ref{th:ElekesErdos} there are constants $c,c_{2},c_{3}>0$, depending only on
$P$, such that $S_{P}(n)\geq cn^{2-c_{2}/\left(  \log n\right)  ^{c_{3}}}$.
Then there is a $(\left\lceil m\right\rceil -1)$-set $A$ such that $\left\vert
A\right\vert +S_{P}(A)\geq cm^{2-c_{2}/\left(  \log m\right)  ^{c_{3}}}$.
Again $A$ does not have $m$ points on a line and setting $c_{1}=2-\log c$ we
get%
\begin{align*}
S_{P}(n,m) &  \geq\frac{1}{I}\left(  \left(  \frac{n}{\left\vert A\right\vert
}\right)  ^{i_{P}(A)}-n\right)  \geq\frac{1}{I}\left(  \left(  \frac{n}%
{m}\right)  ^{2+\frac{\log c}{\log m}-\frac{c_{2}}{\left(  \log m\right)
^{c_{3}}}}-n\right)  \\
&  \geq\frac{1}{Ic}\left(  n^{2-c_{2}/\left(  \log m\right)  ^{c_{3}}%
-c_{1}/\log m}-cn\right)
\end{align*}
for $n$ and $m$ large enough depending only on $P$. That is, $S_{P}%
(n,m)\geq\Omega(n^{2-c_{2}/\left(  \log m\right)  ^{c_{3}}-c_{1}/\log m})$,
where the constant in the $\Omega$ term does not depend on $n$ or $m$.
\end{proof}

If $m$ grows like a fixed power of $n$ and $P$ is cross-algebraic (i.e.,
$S_{P}(n)=\Theta(n^{2})$), then we can improve our bound.

\begin{theorem}
If $P$ is cross-algebraic and $m=m(n)\geq n^{\alpha}$ for some fixed
$0<\alpha<1$, then there is $c_{1}=c_{1}(P,\alpha)>0$ such that
\[
S_{P}(n,m)\geq c_{1}n^{2}\text{ for every }n\geq\left\vert P\right\vert .
\]

\end{theorem}

\begin{proof}
Choose an integer $j\geq2$ such that $\alpha>1/j$. Consider an optimal set $A$
for the function $S_{P}(\left\lfloor n^{1/j}\right\rfloor $. Then $\left\vert
A\right\vert =\left\lfloor n^{1/j}\right\rfloor $ and by Theorem
\ref{th:LaczRuzsa}, there is a constant $c=c(P)$ such that $S_{P}%
(A)=S_{P}(\left\vert A\right\vert )\geq c\left\vert A\right\vert ^{2}$. Since
$\left\vert A\right\vert \leq n^{1/j}<n^{\alpha}\leq m,$ then $A$ has no $m$
collinear points. By Identity (\ref{eq: iteration}) in Theorem \ref{th:main},
\begin{align*}
S_{P}(n,m) &  \geq S_{P}(\left\vert A\right\vert ^{j},m)\geq\frac{1}{I}\left(
\left(  I\cdot S_{P}(A)+\left\vert A\right\vert \right)  ^{j}-\left\vert
A\right\vert ^{j}\right)  \\
&  \geq I^{j-1}S_{P}(A)^{j}\geq c^{j}I^{j-1}\left\vert A\right\vert
^{2j}\text{.}%
\end{align*}
Now, if $n\geq\left\vert P\right\vert $ then $\left\vert A\right\vert \geq
n^{1/j}-1\geq\left(  1-\left\vert P\right\vert ^{-1/j}\right)  n^{1/j}$. By
letting $c_{1}=c^{j}I^{j-1}\left(  1-\left\vert P\right\vert ^{-1/j}\right)
^{2j}$ we get%
\[
S_{P}(n,m)\geq c_{1}n^{2}\text{.}\qedhere
\]

\end{proof}

\section{\label{sec:noparall}Parallelogram-free sets}

We consider the restriction of the function $S_{P}(n)$ to sets of points $A$
in general position (no 3 points on a line) and without parallelograms. We say
that such a set $A$ is \emph{parallelogram-free}. This immediately prohibits
the use of Minkovski Sums to obtain good constructions. More precisely, for a
parallelogram-free pattern $P$, define%
\[
S_{P}^{\nparallel}(n)=\max\left\{  S_{P}(A):\left\vert A\right\vert =n\text{
and }A\text{ is parallelogram-free}\right\}  \text{.}%
\]

We obtain the following upper bound on $S_{P}^{\nparallel}(n)$.
%TCIMACRO{\FRAME{fhFU}{3.5613in}{3.0407in}{0pt}{\Qcb{A parallelogram-free point
%set with $n$ points and $cn\log n$ similar copies of $P$.}}{}{parafree.eps}%
%width 3.5613in;  height 3.0407in;  depth 0pt;  original-width 3.3183in;
%original-height 2.8297in;  cropleft "0";  croptop "1";  cropright "1";
%cropbottom "0";  filename 'parafree.eps';file-properties "XNPEU";}}}%
%BeginExpansion
\begin{figure}
[h]
\begin{center}
\includegraphics[
height=3.0407in,
width=3.5613in
]%
{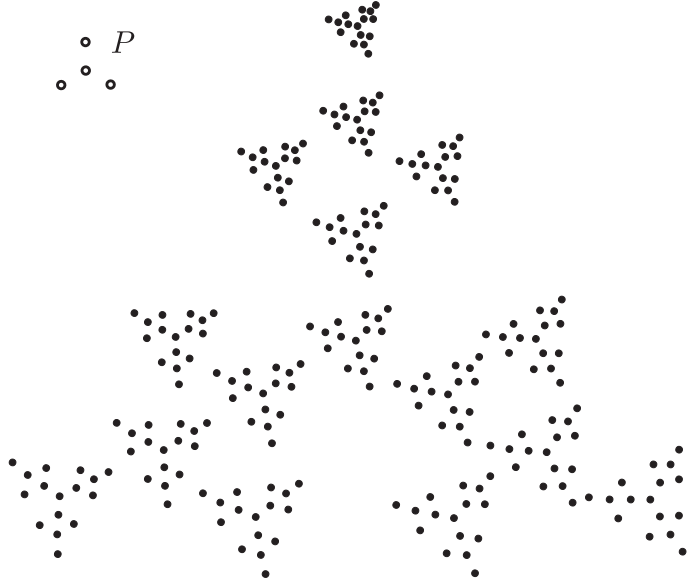}%
\caption{A parallelogram-free point set with $n$ points and $cn\log n$ similar
copies of $P$.}%
\end{center}
\end{figure}
%EndExpansion

\begin{theorem}
Let $P$ be a parallelogram-free pattern with $\left\vert P\right\vert \geq3$.
Then for all $n$,
\[
S_{P}^{\nparallel}(n)\leq n^{3/2}+n\text{.}%
\]

\end{theorem}

\begin{proof}
Suppose $A$ is an $n$-set in the plane in general position and with no
parallelograms. Let $p_{1},p_{2},p_{3}$ be three points in $P$. Consider the
following bipartite graph $B$. The vertex bipartition is $(A,A)$; the edges
are the pairs $(a_{1},a_{2})\in A\times A$ , $a_{1}\neq a_{2}$ such that there
is a point $a_{3}\in A$ with $\triangle a_{1}a_{2}a_{3}\sim\triangle
p_{1}p_{2}p_{3}$. Every similar copy of $P$ in $A$ has at least one edge
$(a_{1},a_{2})$ associated to it. Thus the number of edges $E$ in our graph
satisfies that $E\geq S_{P}(A)$. By a theorem of K\H{o}vari et al. \cite{KST}
(also referred in the literature as Zarankiewicz problem \cite{Z}), it is
known that a bipartite graph with $n$ vertices on each class and without
subgraphs isomorphic to $K_{2,2}$ contains at most $(n-1)n^{1/2}+n$ edges. To
finish our proof we now show that $B$ has no subgraphs isomorphic to $K_{2,2}%
$. Suppose by contradiction that $(a_{1},a_{3}),(a_{1},a_{4}),(a_{2}%
,a_{3}),(a_{2},a_{4})$ are edges in $B$. Let $\lambda=(p_{3}-p_{1}%
)/(p_{2}-p_{1})$. By definition, there are points $a_{13},a_{14},a_{23}%
,a_{24}\in A$ such that $\triangle p_{1}p_{2}p_{3}\sim\triangle a_{1}%
a_{3}a_{13}\sim\triangle a_{1}a_{4}a_{14}\sim\triangle a_{2}a_{3}a_{23}%
\sim\triangle a_{2}a_{4}a_{24}$. Thus $a_{13}=$ $a_{1}+\lambda(a_{3}-a_{1})$,
$a_{14}=$ $a_{1}+\lambda(a_{4}-a_{1})$, $a_{23}=$ $a_{2}+\lambda(a_{3}-a_{2}%
)$, and $a_{24}=$ $a_{2}+\lambda(a_{4}-a_{2})$. Then $a_{13}-a_{14}%
=a_{23}-a_{24}=\lambda(a_{3}-a_{4}),$ which means that $a_{13}a_{14}%
a_{24}a_{23}$ is a parallelogram. This contradicts the parallelogram-free
assumption on $A$.
\end{proof}

We make no attempt to optimize the coefficient of the $n^{3/2}$ term, since we
do not believe that $n^{3/2}$ is the right order of magnitude.

\begin{theorem}
\label{th: para-free sets} Let $P$ be a parallelogram-free pattern with
$\left\vert P\right\vert \geq3$. Then there is a constant $c=c(P)$ such that
for $n\geq\left\vert P\right\vert $,%
\[
S_{P}^{\nparallel}(n)\geq cn\log n\text{.}%
\]

\end{theorem}

For every pattern $P$, we recursively construct a parallelogram-free point set
with many occurrences of $P$. For any $u,v\in\mathbb{C}$, we define%
\begin{equation}
Q(P,A,u,v)=\bigcup_{p\in P}\left(  up+\left(  vp-p+1\right)  A\right)
=\bigcup_{p\in P}\bigcup_{a\in A}\left(  up+\left(  vp-p+1\right)  a\right)
\text{.} \label{eqn:Qdef}%
\end{equation}
Almost all selections of $u$ and $v$ yield a set $Q=Q(P,A,u,v)$ that is
parallelogram-free and such that all the terms in the double union are
pairwise different. The proof of this technical fact is given by next lemma.

\begin{lemma}
\label{lem:tech no para}Let $A$ and $P$ be parallelogram-free sets. If
$\mathcal{S}$ is the set of points $(u,v)\in\mathbb{C}^{2}$ for which
$Q=Q(P,A,u,v)$ satisfies that $\left\vert Q\right\vert <\left\vert
A\right\vert \left\vert P\right\vert $, $Q$ has three collinear points, or $Q$
has a parallelogram; then $\mathcal{S}$ has zero Lebesgue measure.
\end{lemma}

We defer the proof of this lemma and instead proceed to bound the number of
similar copies of $P$ in $Q$.

\begin{lemma}
\label{lem: noparal}If $A$ and $P$ are finite parallelogram-free sets , and
$Q=Q(P,A,u,v)$ defined in (\ref{eqn:Qdef}) satsfies that $\left\vert
Q\right\vert =\left\vert A\right\vert \left\vert P\right\vert $, then%
\[
S_{P}(Q)\geq\left\vert P\right\vert S_{P}(A)+\left\vert A\right\vert \text{.}%
\]

\end{lemma}

\begin{proof}
Because $\left\vert Q\right\vert =\left\vert A\right\vert \left\vert
P\right\vert $ it follows that each term in the first union contributes
exactly $S_{P}(A)$ similar copies of $P$, all of them pairwise different. In
addition note that%
\[
Q=\bigcup_{a\in A}\left(  a+\left(  u+va-a\right)  P\right)  \text{.}%
\]
So each term in the new union is a similar copy of $P$, all of them different
and also different from the ones we had counted before. Therefore
$S_{P}(Q)\geq\left\vert P\right\vert S_{P}(A)+\left\vert A\right\vert $.
\end{proof}

We now prove the theorem.

\begin{proof}
[Proof of Theorem \ref{th: para-free sets}]Let $A_{1}=P$ and for $m\geq1$ let
$A_{m+1}$ be the parallelogram-free set $Q$ obtained from Lemma
\ref{lem:tech no para} with $A=A_{m}$. Because $\left\vert A_{m+1}\right\vert
=\left\vert A_{1}\right\vert \left\vert A_{m}\right\vert $, it follows that
$\left\vert A_{m}\right\vert =\left\vert P\right\vert ^{m}$ for all $m$.
Further, by Lemma \ref{lem: noparal}, for every $0\leq k\leq m-2$,
$S_{P}(A_{m-k})\geq\left\vert P\right\vert S_{P}(A_{m-k-1})+\left\vert
A_{m-k-1}\right\vert =\left\vert P\right\vert S_{P}(A_{m-k-1})+\left\vert
P\right\vert ^{m-k-1}$. Thus%
\begin{align*}
S_{P}(A_{m})  &  \geq\left\vert P\right\vert S_{P}(A_{m-1})+\left\vert
P\right\vert ^{m-1}\geq\left\vert P\right\vert ^{2}S_{P}(A_{m-2})+2\left\vert
P\right\vert ^{m-1}\\
&  \geq\cdots\geq\left\vert P\right\vert ^{m-1}S_{P}(A_{1})+(m-1)\left\vert
P\right\vert ^{m-1}=m\left\vert P\right\vert ^{m-1}\text{.}%
\end{align*}
Now, suppose $\left\vert P\right\vert ^{m}\leq n<\left\vert P\right\vert
^{m+1}$ with $m\geq2$. Let $c=1/(2\left\vert P\right\vert ^{2}\log\left\vert
P\right\vert )$, then%
\[
S_{P}^{\nparallel}(n)\geq S_{P}(A_{m})\geq m\left\vert P\right\vert
^{m-1}>\left(  \frac{\log n}{\log\left\vert P\right\vert }-1\right)  \frac
{n}{\left\vert P\right\vert ^{2}}\geq cn\log n\text{.}%
\]
If $\left\vert P\right\vert \leq n<\left\vert P\right\vert ^{2}$ then
$S_{P}^{\nparallel}(n)\geq S_{P}(P)=1>cn\log n$.
\end{proof}

Finally, we present the proof of Lemma \ref{lem:tech no para}.

\begin{proof}
[Proof of Lemma \ref{lem:tech no para}]We show that $\mathcal{S}$ is made of a
finite number of algebraic sets, all of them of real dimension at most three.
This immediately implies that the Lebesgue measure of such a set is zero. For
every $p\in P$ and $a\in A,$ let $q(a,p)=\left(  up+\left(  vp-p+1\right)
a\right)  $. Suppose that $q(a_{1},p_{1})=q(a_{2},p_{2})$ with $(a_{1}%
,p_{1})\neq(a_{2},p_{2})$. Then
\[
\left(  p_{1}-p_{2}\right)  u+\left(  p_{1}a_{1}-p_{2}a_{2}\right)
v+a_{1}(1-p_{1})-a_{2}\left(  1-p_{2}\right)  =0\text{.}%
\]
This is the equation of a complex-line in $\mathbb{C}^{2}$ (with
real-dimension two) unless the coefficients of $u$ and $v$, as well as the
independent term are equal to zero. That is, $p_{1}-p_{2}=0$, $\left(
p_{1}a_{1}-p_{2}a_{2}\right)  =0$, and $a_{1}(1-p_{1})-a_{2}\left(
1-p_{2}\right)  =0$. These equations imply that $(a_{1},p_{1})=(a_{2},p_{2})$
which contradicts our assumption. Thus the set of pairs $(u,v)$ for which
$\left\vert Q\right\vert <\left\vert A\right\vert \left\vert P\right\vert $ is
the union of $\tbinom{\left\vert A\right\vert \left\vert P\right\vert }{2}$
sets of real-dimension two.

Assume that $q(a_{1},p_{1}),q(a_{2},p_{2})$, and $q(a_{3},p_{3})$ are three
collinear points. Thus there is a real $\lambda\neq0,1$ such that
$q(a_{2},p_{2})-q(a_{1},p_{1})=\lambda\left(  q(a_{3},p_{3})-q(a_{1}%
,p_{1})\right)  $. Then%
\begin{multline*}
\left(  p_{2}-p_{1}-\lambda(p_{3}-p_{1})\right)  u+\left(  p_{2}a_{2}%
-p_{1}a_{1}-\lambda(p_{3}a_{3}-p_{1}a_{1})\right)  v+\\
a_{2}\left(  1-p_{2}\right)  -a_{1}(1-p_{1})-\lambda\left(  a_{3}\left(
1-p_{3}\right)  -a_{1}(1-p_{1})\right)  =0\text{.}%
\end{multline*}
For every $\lambda$, the last equation represents a complex-line in
$\mathbb{C}^{2}$. Considering $\lambda$ as a real variable, this
equation represents an algebraic set of real-dimension three. This
happens unless the coefficients of $u$ and $v$, as well as the
independent term are equal to zero. That is,
$p_{2}-p_{1}-\lambda(p_{3}-p_{1})=0$, $p_{2}a_{2}-p_{1}a_{1}-\lambda
(p_{3}a_{3}-p_{1}a_{1})=0$, and $a_{2}-a_{1}-\lambda\left(  a_{3}%
-a_{1}\right)  =0$. If $p_{1},p_{2},p_{3}$ are three different points then,
since no three points in $P$ are collinear, $p_{2}-p_{1}-\lambda(p_{3}%
-p_{1})\neq0$. If any two of $p_{1},p_{2},p_{3}$ are equal and the remaining
is different, then we still have $p_{2}-p_{1}-\lambda(p_{3}-p_{1})\neq0$. Thus
$p_{1}=p_{2}=p_{3}$, and by symmetry $a_{1}=a_{2}=a_{3}$; which contradicts
the fact that we started with three distinct points $q(a_{j},p_{j})$. Thus the
set of pairs $(u,v)$ for which $\left\vert Q\right\vert $ has collinear points
is the union of $\tbinom{\left\vert A\right\vert \left\vert P\right\vert }{3}$
sets of real-dimension three.

Finally, assume that $q(a_{1},p_{1})q(a_{2},p_{2})q(a_{4},p_{4})q(a_{3}%
,p_{3})$ is a parallelogram. That is $q(a_{2},p_{2})-q(a_{1},p_{1}%
)=q(a_{4},p_{4})-q(a_{3},p_{3})$, and thus%
\begin{multline*}
\left(  p_{2}-p_{1}-\left(  p_{4}-p_{3}\right)  \right)  u+\left(  p_{2}%
a_{2}-p_{1}a_{1}-(p_{4}a_{4}-p_{3}a_{3})\right)  v+\\
a_{2}\left(  1-p_{2}\right)  -a_{1}(1-p_{1})-\left(  a_{4}\left(
1-p_{4}\right)  -a_{3}(1-p_{3})\right)  =0\text{.}%
\end{multline*}
Again this equation represents a complex-line in $\mathbb{C}^{2}$, unless the
coefficients of $u$ and $v$, as well as the independent term are equal to
zero. That is, $p_{2}-p_{1}-\left(  p_{4}-p_{3}\right)  =0$, $p_{2}a_{2}%
-p_{1}a_{1}-(p_{4}a_{4}-p_{3}a_{3})=0$, and $a_{2}-a_{1}-\left(  a_{4}%
-a_{3}\right)  =0$. If $p_{1},p_{2},p_{3},p_{4}$ are four different points
then, since $P$ has no parallelograms, $p_{2}-p_{1}-(p_{4}-p_{3})\neq0$. Since
no three points of $P$ are collinear there are two extra possibilities:
$(p_{1},p_{2})=(p_{3},p_{4})$ or $(p_{1},p_{3})=(p_{2},p_{4})$. By symmetry we
also have $(a_{1},a_{2})=(a_{3},a_{4})$ or $(a_{1},a_{3})=(a_{2},a_{4})$. If
$(p_{1},p_{2})=(p_{3},p_{4})$ and $(a_{1},a_{2})=(a_{3},a_{4})$, then
$q(a_{1},p_{1})=q(a_{3},p_{3});$ which contradicts our assumption. Similarly,
if $(p_{1},p_{3})=(p_{2},p_{4})$ and $(a_{1},a_{3})=(a_{2},a_{4})$, then
$q(a_{1},p_{1})=q(a_{2},p_{2})$. Assume $(p_{1},p_{2})=(p_{3},p_{4})$ and
$(a_{1},a_{3})=(a_{2},a_{4})$. Then the equation $p_{2}a_{2}-p_{1}a_{1}%
-(p_{4}a_{4}-p_{3}a_{3})=0$ becomes $(p_{2}-p_{1})(a_{1}-a_{3})=0$. But if
$p_{1}=p_{2}$ then $q(p_{1},a_{1})=q(p_{2},a_{2})$, and if $a_{1}=a_{3}$ then
$q(p_{1},a_{1})=q(p_{3},a_{3})$; a contradiction in both cases. The remaining
case when $(p_{1},p_{3})=(p_{2},p_{4})$ and $(a_{1},a_{2})=(a_{3},a_{4})$
follows by symmetry. Thus the set of pairs $(u,v)$ for which $\left\vert
Q\right\vert $ has parallelograms is the union of $\tbinom{\left\vert
A\right\vert \left\vert P\right\vert }{4}$ sets of real-dimension two.
\end{proof}

\begin{remark}
If the sets $A$ and $P$ have no two parallel segments then it can be proved,
along the lines of last lemma, that almost all the sets $Q(A,P,u,v)$ are free
of pairs of parallel segments as well.
\end{remark}

\section{Concluding Remarks}

The main relevance of Theorem \ref{th:main} is that it provides an effective
tool to obtain better lower bounds for $S_{P}(n,m)$. Indeed, any of the
results for specific patterns in this paper, can be improved by finding
initial sets with larger indices. For a general pattern $P$, Theorem
\ref{th: general P} is only slightly better than the bound $\Omega
(n^{\log\left(  1+\left\vert P\right\vert \right)  /\log\left\vert
P\right\vert })$ which is obtained using $A=P$ as the initial set. There must
be a way to construct a better initial set.

When the pattern $P=T$ is a triangle we obtained a considerably larger bound
when $T$ is equilateral. The reason behind this is the fact that there is a
multiplying factor of $|\mathrm{Iso}^{+}(T)|=3$ in the index of $i_{T}(A)$. We
could not construct initial sets for arbitrary isosceles triangles that would
improve the bound for scalene triangles. The mirror symmetry of the isosceles
triangles became an obstacle when trying to construct sets with large indices.
For instance, the set $A_{1}$ in Section \ref{sec: scalene} always yields
collinear points when $T$ is an isosceles triangle.

\begin{problem}
For every isosceles triangle $T$, find a set $A$ such that $i_{T}(A)\geq
\log9/\log5$.
\end{problem}

We also had some limitations to construct initial sets when $P=R(k)$ is a
regular polygon. In this case, the index obtained using $P$ itself as initial
set is $\log(2k)/\log k$. We do not have a better initial set for odd $k
\geq7$ and in fact our construction of the initial set $A$ for even-sided
regular $k$-gons in Section \ref{sec: even regular} only gives an index
$i_{R(k)}(A)$ better than $\log(2k)/\log k$ when $k<12$.

\begin{problem}
Let $R(k)$ be a regular $k$-gon. For every even $k\geq12$ and odd $k\geq7$
find a set $A$ such that $i_{R(k)}(A)\geq\log(2k)/\log k$.
\end{problem}

We are confident that there are some yet undiscovered methods for getting
initial sets with larger indices. We would like to find such sets for some
other classes of interesting geometric patterns. For instance, right
triangles, trapezoids, parallelograms, and sets already having some points on
a line, like subsets of a lattice.

According to Theorem \ref{th: m grows}, if we let the number of allowed
collinear points to increase with $n$, then we can achieve $n^{2-\varepsilon}$
similar copies of a pattern $P$. We actually believe this is true even when
$m$ is constant.

\begin{conjecture}
Let $m\geq3$ be a positive integer and $P$ a finite pattern with no $m$
collinear points. For every real $\varepsilon>0,$ there is $N(\varepsilon)>0$
such that for all $n\geq N(\varepsilon)$,%
\[
S_{P}(n,m)\geq n^{2-\varepsilon}\text{.}%
\]

\end{conjecture}

A proof of this conjecture cannot follow from Theorem \ref{th:main}, so such a
proof will require a different way of constructing sets in general position
and with many similar copies of the pattern $P$. On the other hand, we believe
that the construction in Theorem \ref{th: para-free sets} for the function
$S_{P}^{\nparallel}(n)$ is close to optimal. Here we believe that a stronger
upper bound is needed.

\begin{conjecture}
Let $P$ be a parallelogram-free pattern. For every real $\varepsilon>0$, there
is $N(\varepsilon)>0$ such that for all $n\geq N(\varepsilon)$,%
\[
S_{P}^{\nparallel}(n)\leq n^{1+\varepsilon}\text{.}%
\]

\end{conjecture}

\end{document}